\documentclass[final,leqno,letterpaper]{etna}

\setbibdata{1}{xx}{46}{2017} 

\hypersetup{%
    pdftitle={Document title},
    pdfauthor={John Doe, Erwin Schr\"odinger},
    pdfkeywords={a keyword, another keyword, and another one}
    }

\title{Quasi-Monte Carlo and Discontinuous Galerkin\thanks{%
Received... Accepted... Published online on... Recommended by....
}}

\usepackage{graphicx}
\usepackage{pdfsync}
\usepackage{latexsym,amsfonts,amssymb,amsmath,graphics,multirow}
\usepackage{mathrsfs}
\usepackage[normalem]{ulem}
\usepackage[dvipsnames,svgnames,table]{xcolor}
\usepackage{epsfig}
\usepackage{algorithm,algorithmic}
\usepackage{makecell}

\definecolor{darkred}{RGB}{139,0,0}
\definecolor{darkgreen}{RGB}{0,100,0}
\definecolor{darkmagenta}{RGB}{139,0,139}
\definecolor{orange}{RGB}{207,83,0}
\definecolor{brown}{RGB}{139,69,19}

\newcommand{\bsgamma}{{\boldsymbol{\gamma}}}

\newcommand{\bsnu}{{\boldsymbol{\nu}}}

\newcommand{\bsb}{{\boldsymbol{b}}}

\newcommand{\bsm}{{\boldsymbol{m}}}

\newcommand{\bst}{{\boldsymbol{t}}}
\newcommand{\bsx}{{\boldsymbol{x}}}
\newcommand{\bsy}{{\boldsymbol{y}}}

\newcommand{\setu}{\mathrm{\mathfrak{u}}}

\renewcommand*{\vec}{\boldsymbol}
\newcommand*{\elem}{\ensuremath{T}}
\newcommand*{\face}{\ensuremath{F}}
\newcommand*{\setElem}{\ensuremath{\mathcal T_h}}
\newcommand*{\setFace}{\ensuremath{\mathcal F}}
\newcommand*{\dgBil}{\ensuremath{B_h}}
\newcommand*{\poly}{\ensuremath{\mathcal P}}
\newcommand*{\avg}[1]{\ensuremath{\{\![ #1 ]\!\}}}
\newcommand*{\jump}[1]{\ensuremath{[\![ #1 ]\!]}}
\newcommand*{\normal}{\ensuremath{\vec n}}
\newcommand*{\IR}{\ensuremath{\mathbb R}}
\newcommand*{\dx}{\ensuremath{\textup d\vec x}}
\newcommand*{\ds}{\ensuremath{\textup d\sigma}}
\newcommand*{\vecm}{\ensuremath{\vec m}}

\newcommand*{\corr}[1]{{#1}}

\usepackage{tikz}
\newcommand\circlearound[1]{\scalebox{0.8}{\tikz[baseline]\node[circle,inner sep=1pt,draw,anchor=base] {\textup{#1}};}}

\usepackage{pgfplots}
\pgfplotsset{compat=1.15} 
\usetikzlibrary{spy,backgrounds}

\allowdisplaybreaks

\author{Vesa Kaarnioja\footnotemark[2]
        \and Andreas Rupp\footnotemark[3]}

\shorttitle{QUASI-MONTE CARLO AND DISCONTINUOUS GALERKIN} 
\shortauthor{V.~KAARNIOJA AND A.~RUPP}

\begin{document}

\maketitle

\renewcommand{\thefootnote}{\fnsymbol{footnote}}

\footnotetext[2]{Department of Mathematics and Computer Science, Free University of Berlin, Arnimallee~6, DE-14195 Berlin ({\tt vesa.kaarnioja@fu-berlin.de}).}
\footnotetext[3]{Department of Mathematics, Faculty of Mathematics and Computer Science, Saarland University, DE-66123 Saarbr\"{u}cken ({\tt andreas.rupp@uni-saarland.de}).}

\begin{abstract}
In this study, we consider the development of tailored quasi-Monte Carlo (QMC) cubatures for non-conforming discontinuous Galerkin (DG) approximations of elliptic partial differential equations (PDEs) with random coefficients. We consider both the affine and uniform and the lognormal models for the input random field, and investigate the use of QMC cubatures to approximate the expected value of the PDE response subject to input uncertainty. In particular, we prove that the resulting QMC convergence rate for DG approximations behaves in the same way as if continuous finite elements were chosen. Notably, the parametric regularity bounds for DG, which are developed in this work, are also useful for other methods such as sparse grids. Numerical results underline our analytical findings.
\end{abstract}

\begin{keywords}
diffusion equation, discontinuous Galerkin, quasi-Monte Carlo, random coefficient
\end{keywords}

\begin{AMS}
65C05, 65N30
\end{AMS}

\section{Introduction}

In this paper, we consider the design of tailored quasi-Monte Carlo (QMC) rules for a class of non-conforming discontinuous Galerkin (DG) methods used to discretize an elliptic partial differential equation (PDE) with a random diffusion coefficient. More precisely, we investigate the class of interior penalty DG (IPDG) methods and we show that there exist constructible rank-1 lattice rules satisfying rigorous error bounds independently of the stochastic dimension of the problem.

Let $D\subset\mathbb R^d$, $d\in\{1,2,3\}$, be a physical domain with Lipschitz boundary and let $(\Omega,\Gamma,\mathbb P)$ be a probability space. Let us consider the problem of finding $u\!:D\times \Omega\to\mathbb R$ that satisfies 
\begin{align}
\begin{cases}
-\nabla\cdot(a(\bsx,\omega)\nabla u(\bsx,\omega))=f(\bsx),&\bsx\in D,\\
u(\bsx,\omega)=0,&\bsx\in\partial D,
\end{cases}\label{eq:uncertainpde}
\end{align}
for almost all $\omega\in\Omega$. Many studies in uncertainty quantification typically consider one of the following two models for the input random field:

\begin{enumerate}
\item the affine and uniform model (cf., e.g.,~\cite{cds10,spodpaper14,dicklegiaschwab,ghs18,schwabgittelson,kss12,kssmultilevel,schwab13})
$$
a(\bsx,\omega)=a_0(\bsx)+\sum_{j=1}^\infty y_j(\omega)\psi_j(\bsx),\quad \bsx\in D,~\omega\in\Omega,
$$
where $y_1,y_2,\ldots$ are independently and identically distributed random variables with uniform distribution on $[-\tfrac12,\tfrac12]$;
\item the lognormal model (cf., e.g.,~\cite{gittelson,log,log2,log3,log4,log5,schwabtodor})
$$
a(\bsx,\omega)=a_0(\bsx)\exp\bigg(\sum_{j=1}^\infty y_j(\omega)\psi_j(\bsx)\bigg),\quad \bsx\in D,~\omega\in\Omega,
$$
where $y_1,y_2,\ldots$ are independently and identically distributed \corr{standard normal} random variables and $a_0(\bsx)>0$.
\end{enumerate}
Here, $(\psi_j)_{j\geq 1}$ are assumed to be real-valued $L^\infty$ functions on $D$ such that
$$
\sum_{j=1}^\infty \|\psi_j\|_{L^\infty}<\infty.
$$

In the study of uncertainty quantification for PDEs, a natural quantity to investigate is the expected value
$$
\mathbb E[u]:=\int_{\Omega}u(\cdot,\omega)\,{\rm d}\mathbb P(\omega).
$$
In practice, the problem needs to be discretized in several ways before it is possible to approximate this quantity numerically. The infinite-dimensional input random field is first replaced by a finite-dimensional one, meaning that we end up analyzing a dimensionally-truncated PDE solution $u_s$. The dimensionally-truncated solution of the PDE is replaced by a finite element solution $u_{s,h}$, and the high-dimensional integral needs to be approximated by an $s$-dimensional cubature rule with $n$ cubature nodes $Q_{s,n}$. The overall error is then comprised of \emph{dimension truncation error}, \emph{finite element error}, and \emph{cubature error} as
$$
\|\mathbb E[u]-Q_{s,n}(u_{s,h})\|\leq \|\mathbb E[u-u_s]\|+\|\mathbb E[u_s-u_{s,h}]\|+\|\mathbb E[u_{s,h}]-Q_{s,n}(u_{s,h})\|,
$$
for some appropriately chosen norm $\|\cdot\|$. We focus on the cubature error, and discuss the spatial discretization error briefly in Section~\ref{sec:brief}. \corr{We remark that the order of the last two error contributions---the finite element error and cubature error, respectively---can be flipped when the diffusion coefficient is represented using a sequence of bounded random variables (cf., e.g., \cite{kuonuyenssurvey,qmc4pde2,kss12}), but applying this tactic for the lognormal model would require making additional assumptions to ensure that the QMC approximation of $\mathbb E[\|u_{s,h}\|_{H^2}]$ is convergent. Hence, in this paper, we perform the QMC analysis for the dimensionally-truncated, finite element solution of the parametric PDE.}

In recent years, QMC methods have been demonstrated to be very effective at approximating the response statistics of PDE problems such as~\eqref{eq:uncertainpde}. The success of modern QMC theory for uncertainty quantification can largely be attributed to the introduction of weighted spaces (in the sense of Sloan and Wo\'zniakowski~\cite{sloanwoz} and Hickernell~\cite{hick96}) used to analyze the cubature error. While traditional QMC methods can also exhibit faster-than-Monte Carlo convergence rates, as a general rule the convergence rates of classical QMC methods have exponential dependence on the dimension of the integration problem in the unweighted setting~\cite{sloanjoe94}. However, QMC error bounds developed in weighted spaces for integrands satisfying certain moderate smoothness and anisotropy conditions can be shown to be independent of the stochastic dimension while retaining their faster-than-Monte Carlo convergence rates. For a detailed survey about the development of QMC methods over the past few decades, see~\cite{dksacta}.

QMC methods are well-suited for large scale computations since they are trivial to parallelize. Other advantages include the possibility to generate the cubature point sets for certain classes of QMC methods, such as randomly shifted rank-1 lattice rules or digital nets, on the fly. Randomization of QMC rules also enables the computation of practical error estimates. These features make QMC methods ideal for heavy-duty uncertainty quantification compared to regular Monte Carlo methods (slow convergence rate) or sparse grids (not easily parallelizable). QMC methods have been applied successfully to many problems such as Bayesian inverse problems~\cite{bayes2,bayes3,bayes4,bayes}, spectral eigenvalue problems~\cite{spectral}, optimization under uncertainty~\cite{ocpde,guth22,mlqmc}, the Schr\"odinger equation~\cite{schrodinger}, the wave equation~\cite{wave}, problems arising in quantum physics~\cite{quantum}, and various others.

QMC error bounds in the weighted space setting have the following generic form
$$
\text{root mean squared error}\leq C_{\boldsymbol\gamma,n,s}\|G\|_{s,\boldsymbol\gamma},
$$
where $G\colon [0,1]^s\to\mathbb R$ denotes the integrand, $\boldsymbol\gamma=(\gamma_{\setu})_{\setu\subseteq\{1,\ldots,s\}}$ denotes a collection of positive weights, $C_{\boldsymbol\gamma,n,s}>0$ denotes the ({\em shift-averaged}) {\em worst case error} depending on the weights $\boldsymbol\gamma$, the number of QMC nodes $n$, and the dimension $s$, while $\|\cdot\|_{s,\boldsymbol\gamma}$ denotes the norm of an appropriately chosen weighted Sobolev space. The analysis of the cubature error for PDE uncertainty quantification typically consists of the following steps: first, the parametric regularity of the integrand needs to be estimated by deriving {\em a priori} upper bounds on the higher-order partial derivatives $\partial^{\bsnu}G$ with respect to the (uncertain) variables. In consequence, these {\em a priori} upper bounds can then be used to estimate the weighted Sobolev norm $\|G\|_{s,\boldsymbol\gamma}$ of the integrand. Finally, the weights $\boldsymbol\gamma$ in the QMC error bound serve essentially as free parameters and can be chosen freely to optimize the cubature error bound. Typically, one chooses the weights in such a way that $\corr{C_{\boldsymbol\gamma,n,s}\|G\|_{s,\boldsymbol\gamma}}\leq C_{\boldsymbol\gamma,n}<\infty$ for some constant $C_{\boldsymbol\gamma,n}>0$ not dependent on $s$, ensuring that the error rate of the resulting cubature rule can be bounded independently of the dimension $s$. From a practical point of view, the weights that minimize the error bound can be used as inputs to the so-called \emph{component-by-component (CBC) algorithm}~\cite{cbc1,cbc2} which produces a QMC rule satisfying the theoretically derived error bound.  For a detailed description of this procedure, the reader is referred to the surveys~\cite{kuonuyenssurvey,qmc4pde2}. There also exist freely available software~\cite{qmcpy,qmc4pde} for the construction of QMC point sets.

A common feature of virtually all of the aforementioned QMC literature related to PDE uncertainty quantification is that the QMC rule is designed for the \emph{non-discretized} PDE problem~\eqref{eq:uncertainpde} whereas, in practical computations, one only has access to a discrete approximation of the PDE system. Of course, as long as one uses a conforming finite element (FE) method to perform the discretization of the PDE, \corr{the parametric regularity of the non-discretized PDE naturally transfers} to the discretized PDE problem and the theory remains sound even for the discretized PDE system. However, in many cases it is preferable to consider \emph{non-conforming FE methods} such as DG methods (e.g., see the books \cite{PietroE12,DolejsiF15,HesthavenW07,KnabnerA21,Riviere08} for a comprehensive overview and \cite{ArnoldBCM02} for a unified analysis framework for elliptic PDEs using DG) to solve PDEs with uncertain coefficients. Using DG methods, the inter-element continuity constraint of conforming FE methods is dropped and concatenations of arbitrary local polynomials with support in only one element can be used as test and trial functions. This property resembles the finite volume (FV) approach, in which the solution per element is approximated by a local constant. As a consequence, DG methods have similar parallelization properties as FV (but do not need any reconstructions to achieve higher order). Moreover, the problem of hanging nodes is intrinsically bypassed in the DG framework allowing for more general meshes than conforming FE, see \cite{CangianiGH14}. Beyond that, DG directly supports $hp$ refinement, where both the mesh and the local degree of approximating polynomials can be adapted locally, see \cite{CangianiGH14,CastilloCSS02,PerugiaS02}. Last, DG comes with an intuitive notion of local mass conservation, since it is based on inter-element fluxes (and mass conservative in an element itself). These advantages have made DG methods a popular tool for many applications. In particular, they have been used in the field of computational fluid dynamics (see \cite{Shu13} for a survey---notable mentions are DG methods for incompressible Navier--Stokes \cite{CockburnKS05} and shallow-water equations \cite{AizingerD07}), multicomponent reactive transport \cite{SunW06}, and many more.
However, since discontinuous Galerkin methods are non-conforming, we cannot directly apply the existing QMC theory. 

This paper tries to bridge this theoretical gap. It is structured as follows. Notations and preliminaries are introduced in Section~\ref{sec:notations}. Section~\ref{SEC:qmcc} describes randomly shifted lattice rules, and Section~\ref{SEC:conforming_fem} gives a brief overview over the analysis of conforming FE methods. DG in the QMC framework is presented in Section~\ref{SEC:dg} and the corresponding parametric regularity analysis is considered in Section~\ref{sec:parametric}. Numerical results, which confirm our theoretical findings, are given in Section~\ref{SEC:numerics}, while a short conclusion wraps up our exposition.

\section{Notations, preliminaries, and assumptions}\label{sec:notations}
Let expression $\mathscr F:=\{\bsnu\in\mathbb N_0^{\mathbb N}\mid |{\rm supp}(\bsnu)|<\infty\}$ denote the set of all multi-indices with finite support ${\rm supp}(\bsnu):=\{j\in\mathbb N\mid \nu_j\neq 0\}$. Moreover, we denote the order of a multi-index $\bsnu=(\nu_1,\nu_2,\ldots)\in\mathscr F$ by
$$
|\bsnu|:=\sum_{j\in{\rm supp}(\bsnu)}\nu_j.
$$
We will also use the shorthand notation $\{1:s\}:=\{1,\ldots,s\}$.

The relevant function space for our PDE problem will be $V:=H_0^1(D)$. Its dual space $V'=H^{-1}(D)$ is understood with respect to the pivot space $H:=L^2(D)$, which we identify with its own dual. We define
$$
\|v\|_V:=\|\nabla v\|_{H}.
$$

Let $U$ be a space of parameters, and otherwise for the moment not specified. We consider the variational formulation of the PDE problem~\eqref{eq:uncertainpde}: for all $\bsy\in U$, find $u(\cdot,\bsy)\in V$ such that
\begin{align}
\int_D a(\bsx,\bsy)\nabla u(\bsx,\bsy)\cdot \nabla v(\bsx)\,{\rm d}\bsx=\langle f, v\rangle\quad\text{for all}~v\in V,\label{eq:weak}
\end{align}
where $f\in V'$ and $\langle\cdot,\cdot\rangle\corr{\,:=\langle \cdot,\cdot\rangle_{V',V}}$ denotes the duality pairing \corr{of $V$ and $V'$}.
\paragraph{Uniform and affine setting} In the uniform and affine setting, we fix the set of parameters to be $U:=[-\frac12,\frac12]^{\mathbb N}$, define
$$
a(\bsx,\bsy):=a_0(\bsx)+\sum_{j=1}^\infty y_j \psi_j(\bsx),\quad \bsx\in D,~\bsy\in U,
$$
and assume the following:
\begin{itemize}
\item[(U1)] $a_0\in L^\infty(D)$ and $\sum_{j=1}^\infty \|\psi_j\|_{L^\infty}<\infty$;
\item[(U2)] there exist $a_{\min},a_{\max}>0$ such that $a_{\min}\leq a(\bsx,\bsy)\leq a_{\max}$ for all $\bsx\in D$ and $\bsy\in U$;
\item[(U3)] $\sum_{j=1}^\infty \|\psi_j\|_{L^\infty}^p<\infty$ for some $0<p<1$;
\item[(U4)] $a_0\in W^{1,\infty}(D)$ and $\sum_{j=1}^\infty \|\psi_j\|_{W^{1,\infty}}<\infty$, where
$$
\|v\|_{W^{1,\infty}}:=\max\{\|v\|_{L^\infty},\|\nabla v\|_{L^\infty}\};
$$
\item[(U5)] the spatial domain $D\subset\mathbb R^d$, $d\in\{1,2,3\}$, is a convex and bounded polyhedron.
\end{itemize}
We consider the expected value
$$
\mathbb E[u]=\int_U u(\cdot,\bsy)\,{\rm d}\bsy.
$$

\paragraph{Lognormal setting} In the lognormal setting, we fix the set of (admissible) parameters to be $U:=U_{\boldsymbol\beta}:=\{\bsy\in\mathbb R^{\mathbb N}:\sum_{j\geq 1}\beta_j|y_j|<\infty\}$, where $\boldsymbol\beta:=(\beta_j)_{j\geq 1}$ and $\beta_j:=\|\psi_j\|_{L^\infty}$. \corr{We model the input coefficient as}
\begin{align}
a(\bsx,\bsy):=a_0(\bsx)\exp\bigg(\sum_{j=1}^\infty y_j\psi_j(\bsx)\bigg),\quad \bsx\in D,~\bsy\in U,\label{eq:lognormalmodel}
\end{align}
and assume the following:
\begin{itemize}
\item[(L1)] $a_0\in L^\infty(D)$ and $\sum_{j=1}^\infty \|\psi_j\|_{L^\infty}<\infty$;
\item[\corr{(L2)}]\corr{there hold $\min_{\bsx\in\overline{D}}a_0(\bsx)>0$ and $\max_{\bsx\in\overline{D}}a_0(\bsx)<\infty$, and we define
\begin{align*}
&a_{\min}(\bsy):=\big(\min_{\bsx\in \overline{D}}a_0(\bsx)\big)\exp\bigg(-\sum_{j=1}^\infty |y_j|b_j\bigg),\\
&a_{\max}(\bsy):=\big(\max_{\bsx\in \overline{D}}a_0(\bsx)\big)\exp\bigg(\sum_{j=1}^\infty |y_j|b_j\bigg),
\end{align*}
for all $\bsy\in U_{\boldsymbol\beta}$.}
\item[(L3)] $\sum_{j=1}^\infty \|\psi_j\|_{L^\infty}^p<\infty$ for some $0<p<1$;
\item[(L4)] $a_0\in W^{1,\infty}(D)$ and $\sum_{j=1}^\infty \|\psi_j\|_{W^{1,\infty}}<\infty$;
\item[(L5)] the spatial domain $D\subset\mathbb R^d$, $d\in\{1,2,3\}$, is a convex and bounded polyhedron.
\end{itemize}
\corr{Since the set $U$ has full measure with respect to the infinite-dimensional product Gaussian measure $\mu_G=\bigotimes_{j=1}^\infty\mathcal N(0,1)$, we can identify~\cite[Lem.~2.28]{schwabgittelson}}%
$$
\mathbb E[u]=\int_{\mathbb R^{\mathbb N}} u(\cdot,\bsy)\,\mu_G({\rm d}\bsy)=\int_{U} u(\cdot,\bsy)\,\mu_G({\rm d}\bsy).%
$$
\corr{Thus it is sufficient to constrain our parametric regularity analysis of $u(\cdot,\bsy)$ to the parameters $\bsy\in U$ for which the parametric PDE problem is well-defined.}

{\em Remark.} (i) Let $(\Omega,\Gamma,\mathbb P)$ be a probability space. Every square-integrable Gaussian random field $Z\!:D\times\Omega\to\mathbb R$ with a continuous, symmetric, and positive definite covariance function admits a Karhunen--Lo\`eve expansion
$$
Z(\bsx,\omega)=\sum_{j=1}^\infty \sqrt{\lambda_j}y_j(\omega)\psi_j(\bsx),\quad y_j\overset{\rm i.i.d.}{\sim}\mathcal N(0,1),
$$
where $\{\psi_j\}_{j=1}^\infty$ constitutes an orthonormal basis for $L^2(D)$ and $\lambda_1\geq\lambda_2\geq\cdots\geq 0$ with $\lim_{j\to\infty}\lambda_j=0$. However, not all random fields with such expansions are Gaussian. Hence we concentrate on a broader class of random fields characterized by expansions of the form~\eqref{eq:lognormalmodel} which we refer to as the ``lognormal model'' for brevity.

\corr{(ii) Assumption (L2) together with~\cite[Lem.~3.10]{gittelson} ensure that both $a_{\max}(\cdot)$ and $1/a_{\min}(\cdot)$ are integrable with respect to the Gaussian product measure $\bigotimes_{j\geq 1}\mathcal N(0,1)$.}%
\section{Quasi-Monte Carlo cubature}\label{SEC:qmcc}
Since QMC methods can only be applied to finite-dimensional integrals and the analysis of the dimension truncation error is independent of the chosen spatial discretization scheme (cf., e.g.,~\cite{Gantner,kuonuyenssurvey}), we restrict our analysis to the finite-dimensional setting in what follows.
\subsection{Uniform and affine model}
We are interested in solving an $s$-dimensional integration problem
\begin{align*}
I_s(G):=\int_{[-1/2,1/2]^s}G(\bsy)\,{\rm d}\bsy.
\end{align*}
As our QMC estimator of $I_s(G)$, we take
$$
Q_{\rm ran}(G):=\frac{1}{nR}\sum_{i=1}^n \sum_{r=1}^R G(\{\bst_i+\boldsymbol\Delta^{(r)}\}-\tfrac{\boldsymbol 1}{\boldsymbol 2}),
$$
where \corr{$\boldsymbol\Delta^{(1)},\ldots,\boldsymbol\Delta^{(R)}$ are i.i.d.~realizations of a random variable $\boldsymbol\Delta\sim U([0,1]^s)$}, $\{\cdot\}$ denotes the componentwise fractional part, $\frac{\mathbf 1}{\mathbf 2}:=[\tfrac12,\ldots,\tfrac12]^{\rm T}\in\mathbb R^s$, and
\begin{align}
\boldsymbol{t}_i:=\bigg\{\frac{i\boldsymbol{z}}{n}\bigg\}\quad\text{for}~i\in\{1,\ldots,n\},\label{eq:qmcnode}
\end{align}
where $\boldsymbol{z}\in\{0,\ldots,n-1\}^s$ is called the \emph{generating vector}.%

Let us assume that the integrand $G$ belongs to a weighted, unanchored Sobolev space with bounded first order mixed partial derivatives, the norm of which is given by
$$
\|G\|_{s,\boldsymbol \gamma}^2:=\sum_{\setu\subseteq\{1:s\}}\frac{1}{\gamma_{\setu}}\int_{[-1/2,1/2]^{|\setu|}}\bigg(\int_{[-1/2,1/2]^{s-|\setu|}}\frac{\partial^{|\setu|}}{\partial \bsy_{\setu}}G(\bsy)\,{\rm d}\bsy_{-\setu}\bigg)^2\,{\rm d}\bsy_{\setu}.
$$
Here $\bsgamma:=(\gamma_{\setu})_{\setu\subseteq\{1:s\}}$ is a collection of positive weights, ${\rm d}\bsy_{\mathfrak u}:=\prod_{j\in\mathfrak u}{\rm d}y_j$, and ${\rm d}\bsy_{-\mathfrak u}:=\prod_{j\in\{1:s\}\setminus\mathfrak u}{\rm d}y_j$ for $\mathfrak u\subseteq\{1:s\}.$

The following well-known result shows that it is possible to construct a generating vector satisfying a rigorous error bound using a CBC algorithm~\cite{cbc2} (see also~\cite{dksacta}).

\begin{lemma}[cf.~{\cite[Thm.~5.1]{kuonuyenssurvey}}]\label{lemma:affineqmc}
Let $G$ belong to the weighted unanchored Sobolev space over $[0,1]^s$ with weights $\boldsymbol{\gamma}=(\gamma_{\setu})_{\setu\subseteq\{1:s\}}$. A randomly shifted lattice rule with $n=2^m$ points in $s$ dimensions can be constructed by  a CBC algorithm such that for $R$ independent random shifts and for all $\lambda\in (1/2,1]$, it holds that
$$
\sqrt{\mathbb{E}_{\boldsymbol{\Delta}}[|I_{s}(G)-Q_{\rm ran}(G)|^2]}\leq\frac{1}{\sqrt{R}}\bigg(\frac{2}{n}\sum_{\emptyset\neq \setu\subseteq\{1:s\}}\gamma_{\setu}^\lambda\varrho(\lambda)^{|\setu|}\bigg)^{1/(2\lambda)}\|G\|_{s,\bsgamma},
$$
where
$$
\varrho(\lambda):=\frac{2\zeta(2\lambda)}{(2\pi^2)^\lambda}.
$$
Here, $\zeta(x):=\sum_{k=1}^\infty k^{-x}$ is the \emph{Riemann zeta function} for $x>1$ and $\mathbb E_{\boldsymbol\Delta}[\cdot]$ denotes the expected value with respect to uniformly distributed random shift over $[0,1]^s$.
\end{lemma}
\subsection{Lognormal model}
We are interested in solving an $s$-dimensional integration problem
\begin{align*}
I_s^{\varphi}(G):=\int_{\mathbb{R}^s}G(\bsy)\prod_{j=1}^s\varphi(y_j)\,{\rm d}\bsy=\int_{[0,1]^s}G(\Phi^{-1}(\boldsymbol{w}))\,{\rm d}\boldsymbol{w},%
\end{align*}
where $\varphi\!:\mathbb{R}\to\mathbb{R}_+$ is the probability density function of the standard normal distribution, i.e.,
\begin{align}
\varphi(x):=\frac{1}{\sqrt{2\pi}}\exp\bigg(-\frac12 x^2\bigg),\quad x\in\mathbb{R},\label{eq:normalpdf}
\end{align}
and $\Phi^{-1}$ denotes the inverse cumulative distribution function of $\prod_{j=1}^s\varphi(y_j)$.%

As our QMC estimator of $I_s(G)$, we take
$$
Q_{\rm ran}^{\varphi}(G):=\frac{1}{nR}\sum_{i=1}^n\sum_{r=1}^R G(\Phi^{-1}(\{\bst_i+\boldsymbol\Delta^{(r)}\})),
$$
where \corr{$\boldsymbol\Delta^{(1)},\ldots,\boldsymbol\Delta^{(R)}$ are i.i.d.~realizations of a random variable $\boldsymbol\Delta\sim U([0,1]^s)$} and the nodes $\boldsymbol{t}_1,\ldots,\boldsymbol{t}_n\in[0,1]^s$ are defined as in~\eqref{eq:qmcnode}.

Let us assume that the integrand $G$ belongs to a special weighted Sobolev space in $\mathbb{R}^s$ with bounded first order mixed partial derivatives, the norm of which is given by
\begin{multline*}
\|G\|_{s,\bsgamma}^2:= \\ 
\sum_{\setu\subseteq\{1:s\}}\!\frac{1}{\gamma_{\setu}}\!\int_{\mathbb{R}^{|\setu|}}\!\bigg(\int_{\mathbb{R}^{s-|\setu|}}\frac{\partial^{|\setu|}}{\partial \bsy_{\setu}}G(\bsy)\bigg(\prod_{j\in\{1:s\}\setminus\setu}\!\varphi(y_j)\bigg)\,{\rm d}\bsy_{-\setu}\bigg)^2\bigg(\prod_{j\in \setu}\varpi_j^2(y_j)\bigg)\,{\rm d}\bsy_{\setu},
\end{multline*}
where $\bsgamma:=(\gamma_{\setu})_{\setu\subseteq\{1:s\}}$ is a collection of positive weights and we define the weight functions
\begin{align}
\varpi_j(x):=\exp(-\alpha_j\,|x|),\quad \alpha_j>0,\quad j\in\mathbb{Z}_+.\label{eq:varpi}
\end{align}
\corr{The parameters $\alpha_j>0$ in~\eqref{eq:varpi} are for the moment arbitrary, but we will end up fixing their values later on in order to obtain dimension-independent QMC convergence for the lognormal model problem.}

In analogy to the affine and uniform setting, the following well-known result gives an error bound for a QMC rule based on a generating vector constructed using the CBC algorithm.
\begin{lemma}[cf.~{\cite{nicholskuo} and \cite[Thm.~15]{log}}]
Let $G$ belong to the weighted function space over $\mathbb{R}^s$ with weights $\boldsymbol{\gamma}=(\gamma_{\setu})_{\setu\subseteq\{1:s\}}$, let $\varphi\!:\mathbb{R}\to\mathbb{R}_+$ be the standard normal density defined by~\eqref{eq:normalpdf}, and let the weight functions $\varpi_j$ be defined by~\eqref{eq:varpi}. A randomly shifted lattice rule with $n=2^m$ points in $s$ dimensions can be constructed by  a CBC algorithm such that for $R$ independent random shifts and for all $\lambda\in (1/2,1]$, it holds that
$$
\sqrt{\mathbb{E}_{\boldsymbol{\Delta}}[|I_{s}^{\varphi}(G)-Q_{\rm ran}^{\varphi}(G)|^2]}\leq\frac{1}{\sqrt{R}}\bigg(\frac{2}{n}\sum_{\emptyset\neq \setu\subseteq\{1:s\}}\gamma_{\setu}^\lambda\prod_{j\in\setu}\varrho_j(\lambda)\bigg)^{1/(2\lambda)}\|G\|_{s,\bsgamma},
$$
where
$$
\varrho_j(\lambda):=2\bigg(\frac{\sqrt{2\pi}\exp(\alpha_j^2/\eta_{\ast})}{\pi^{2-2\eta_{\ast}}(1-\eta_{\ast})\eta_{\ast}}\bigg)^{\lambda}\zeta(\lambda+\tfrac12),\quad \eta_{\ast}:=\frac{2\lambda-1}{4\lambda}.
$$
Here, $\zeta(x):=\sum_{k=1}^\infty k^{-x}$ is the \emph{Riemann zeta function} for $x>1$ and $\mathbb E_{\boldsymbol\Delta}[\cdot]$ denotes the expected value with respect to uniformly distributed random shift over $[0,1]^s$.
\end{lemma}
\section{Conforming FE methods}\label{SEC:conforming_fem}
In the case of conforming finite element discretizations for the elliptic PDE problem, it is enough to analyze the parametric regularity of the continuous problem. The parametric regularity results are inherited by the conforming FE solution. Below, we briefly recap the main parametric regularity results for the affine and uniform as well as the lognormal model.

Let us consider the weak formulation~\eqref{eq:weak}: for all $\bsy\in U$, find $u(\cdot,\bsy)\in V$ such that
\begin{align}
\int_D a(\bsx,\bsy)\nabla u(\bsx,\bsy)\cdot \nabla v(\bsx)\,{\rm d}\bsx=\langle f, v\rangle\quad\text{for all}~v\in V.\label{eq:bilinform}
\end{align}
If we use a conforming FE method to solve the equation above, the variational formulation stays the same
$$
\int_D a(\bsx,\bsy)\nabla u_h(\bsx,\bsy)\cdot \nabla v_h(\bsx)\,{\rm d}\bsx=\langle f, v_h \rangle\quad\text{for all}~v_h\in V_h,
$$
where the test and trial space $V_h$ is a finite dimensional subspace of $V$.

\paragraph{Uniform and affine setting}
\corr{It was shown in~\cite{cds10}} that the solution to~\eqref{eq:bilinform} satisfies the regularity bounds
\begin{equation}\label{EQ:reg_bound_fe_aff}
\|\partial_{\bsy}^{\bsnu}u(\cdot,\bsy)\|_V\leq |\bsnu|!\,\bsb^{\bsnu}\,\frac{\|f\|_{V'}}{a_{\rm min}},\quad \bsb:=(b_j)_{j\geq 1},~b_j:=\frac{\|\psi_j\|_{L^\infty}}{a_{\min}},
\end{equation}
for all finitely supported multi-indices $\bsnu\in\mathcal F$ and $\bsy\in U$. \corr{Since $V_h$ is now a subspace of $V$, it follows}  that the same regularity bound holds for the dimensionally-truncated FE approximation:
\begin{equation}\label{EQ:reg_bound_fe_log}
\|\partial_{\bsy}^{\bsnu}u_{s,h}(\cdot,\bsy)\|_V\leq |\bsnu|!\,\bsb^{\bsnu}\,\frac{\|f\|_{V'}}{a_{\rm min}}.
\end{equation}

\paragraph{Lognormal setting}
The corresponding regularity bound for the lognormal case is~\corr{\cite{log}}
$$
\|\partial_{\bsy}^{\bsnu}u(\cdot,\bsy)\|_V\leq \frac{|\bsnu|!}{(\log 2)^{|\bsnu|}}\,{\boldsymbol\beta}^{\bsnu}\,\frac{\|f\|_{V'}}{a_{\rm min}(\bsy)},\quad \boldsymbol\beta:=(\beta_j)_{j\geq 1},~\beta_j:=\|\psi_j\|_{L^\infty},
$$
for all finitely supported multi-indices $\bsnu\in\mathcal F$ and $\bsy\in U_{\boldsymbol\beta}$. Note that $u$ can be again replaced by a dimensionally-truncated, conforming FE approximation $u_{s,h}$ in the above inequality.

\section{Discontinuous Galerkin FE methods}\label{SEC:dg}
In this section, we analyze the DG method provided that $D \subset \mathbb R^d$ is a convex polyhedron, $\emptyset \neq U \subset \mathbb R^\mathbb N$, $f \in L^2(D)$, and that there are
\begin{equation*}
 0 < a_{\min}(\vec y) = \min_{x \in \bar D} a(\vec x, \vec y) \quad \text{ and } \quad a_{\max}(\vec y) = \max_{x \in \bar D} a(\vec x, \vec y) < \infty.
\end{equation*}
This setup holds for both the uniform and affine and the lognormal cases.

In DG, the idea is to modify the variational formulation, and we can no longer exploit conformity to obtain regularity bounds for the DG solutions to~\eqref{eq:uncertainpde}. This means that the regularity analysis must be rewritten for the DG system, which will also affect the choice of the optimal QMC rule for the computation of \corr{the expectation of the stochastic response.}

The most prominent DG method for model problem \eqref{eq:uncertainpde} is given by the family of interior penalty (IPDG) methods. For these methods, we define
\begin{equation*}
 V_h = \{ v_h \in L^2(D) \colon v_h|_\elem \in \poly_k(\elem) \; \forall \elem \in \setElem \},
\end{equation*}
where $\poly_k(\elem)$ denotes the space of (multivariate) polynomials of (total) degree at most $k$ on $\elem$, and $\setElem$ is supposed to be a member of a shape and contact regular mesh sequence in the sense of \cite[Def.\ 1.38]{PietroE12}. Basically, this says that the mesh elements are neither distorted nor have arbitrary small angles. Purely for the sake of simplicity in the notation, we assume the mesh to be geometrically conforming (have no hanging nodes). Importantly, we do not require our mesh to be simplicial or quadrilateral, etc. In fact, DG methods can easily be used on significantly more general meshes. Note that $V_h$ does not impose any inter-element continuity constraints. Thus, we also define the mean value $\avg{\cdot}$ and the jump $\jump{\cdot}$ on a common face $\face$ of elements $\elem^+ \neq \elem^- \in \setElem$ by
\begin{equation*}
 \avg{ v_h } = \frac{1}{2} \left( v_h|_{\elem^+} + v_h|_{\elem^-} \right) \qquad \text{ and } \qquad \jump{v_h} = v_h|_{\elem^+} \normal_{\elem^+} + v_h|_{\elem^-} \normal_{\elem^-},
\end{equation*}
where $\normal_\elem$ denotes the outward pointing unit normal with respect to $\elem$. Importantly, the jump turns a scalar function to a vector function. Likewise, the average of a vector quantity is defined component wise, and if $\face \subset \partial D$, we set
\begin{equation*}
 \avg{ v_h } = v_h|_{\elem} \qquad \text{ and } \qquad \jump{v_h} = v_h|_{\elem} \normal,
\end{equation*}
where $\elem$ refers to the unique $\elem \in \setElem$ with $\face \subset \partial \elem$ and $\normal$ is the outward pointing unit normal of $D$. The set of all faces is denoted by $\setFace$.

Using these definitions, we can define the DG bilinear form $\dgBil\colon V_h \times V_h \to \IR$ as
\begin{align}
 \dgBil&(\vec y; u_h(\cdot, \vec y), v_h)\notag \\
 = & \!\sum_{\elem \in \setElem} \!\int_\elem \!a(\vec x, \vec y) \nabla u_h(\vec x, \vec y) \!\cdot\! \nabla v_h(\vec x) \dx \!+\! \!\sum_{\face \in \setFace} \!\int_\face \!\bigg[ \theta \avg{a(\vec x, \vec y) \nabla v_h(\vec x)} \! \cdot\!\! \jump{u_h(\vec x, \vec y)}\notag \\
 & \qquad - \avg{a(\vec x, \vec y) \nabla u_h(\vec x, \vec y)} \cdot \jump{v_h(\vec x)} + \frac{\eta(\vec y)}{h_\face} \jump{ u_h(\vec x, \vec y) } \cdot \jump{v_h(x)} \bigg] \ds\notag \\
 = & \sum_{\elem \in \setElem} \int_\elem a(\vec x, \vec y) \nabla u_h(\vec x, \vec y) \cdot \nabla v_h(\vec x) \dx\label{eq:refertome} \\
 & \quad + \sum_{\elem \in \setElem} \int_{\partial \elem} \bigg[ \theta \avg{a(\vec x, \vec y) \nabla v_h(\vec x)}  \cdot u_h(\vec x, \vec y) \normal_\elem(\vec x)\notag \\
 & \qquad - \frac{1}{2} a(\vec x, \vec y) \nabla u_h(\vec x, \vec y) \cdot \jump{v_h(\vec x)} + \frac{\eta(\vec y)}{h_\face} u_h(\vec x, \vec y) \normal_\elem(\vec x) \cdot \jump{v_h(x)} \bigg] \ds\notag
\end{align}
and the corresponding linear form $\mathcal L\colon V_h \to \IR$ as
\begin{equation}\label{EQ:rhs}
 \mathcal L(v_h) = \int_D f(\vec x) v_h(\vec x) \dx.
\end{equation}
Here, $h_\face$ denotes the diameter of the considered face. The parameter $\theta \in \{-1,0,+1\}$ is chosen to be equal to $1$ for the non-symmetric IP method (NIPDG), equal to $0$ for the incomplete IP method (IIPDG), and equal to $-1$ for the symmetric IP method (SIPDG). The role of parameter $\eta(\bsy)$ is to penalize jumps in $u_h$ across interfaces, and thereby stabilize the method. It is well known that the NIPDG method is stable for parameters $\eta(\vec y) \ge 0$, while IIPDG and SIPDG methods are stable \corr{if $\eta(\vec y) \ge \bar \eta(\vec y) > 0$. Here, $\bar \eta(\vec y)$ depends, among others, on the infimum and supremum of $a(\cdot, \vec y)$, see Lemma \ref{LEM:dg_stability}. Thus, we cannot expect to choose $\eta$ independent of $\vec y$ for the IIPDG and SIPDG methods in the lognormal case.} The NIPDG method for $\eta(\vec y) = 0$ is attributed to Oden, Babu\v ska, and Baumann \cite{OdenBB98}, but we exclude it from our considerations (since its analysis is more involved), and assume that $\eta(\vec y) > 0$ is used. For a detailed discussion of these methods, the reader is referred to the books \cite{PietroE12,DolejsiF15,HesthavenW07,KnabnerA21,Riviere08} and the references therein.
 
\subsection{Preliminaries}\label{sec:stability}
Notably, $V_h \not\subset H^1(D)$, and thus we have to define a norm that allows for a generalization of Poincar\'e's inequality. We choose
\begin{equation*}
 \| v_h \|^2_{V_h} = \sum_{\elem \in \setElem} \| \sqrt{ a(\cdot , \vec y) } \nabla v_h(\cdot) \|^2_{L^2(\elem)} + \sum_{\face \in \setFace} \frac{\eta(\vec y)}{h_\face} \| \jump{v_h} \|^2_{L^2(\face)}.
\end{equation*}
Norm $\| \cdot \|_{V_h}$ actually also depends on $\vec y$, which is suppressed in the notation. In order to deduce Poincar\'e's inequality, we cite \cite[Cor.\ 5.4]{PietroE12}:
\begin{lemma}
 \corr{Let $\mathcal T_h$ be a shape and contact regular mesh sequence member.} For the norm
 \begin{equation*}
  \| v_h \|_\textup{dG} = \left( \sum_{\elem \in \setElem} \| \nabla v_h \|^2_{L^2(\elem)} + \sum_{\face \in \setFace} \frac{1}{h_\face} \| \jump{v_h} \|^2_{L^2(\face)}  \right)^{1/2}
 \end{equation*}
 there is a constant $\sigma > 0$, \corr{independent of h,} such that
 \begin{equation*}
  \forall v_h \in V_h, \qquad \| v_h \|_{L^2(D)} \le \sigma \| v_h \|_\textup{dG}
 \end{equation*}
 holds.
\end{lemma}

\corr{Here, the notion ``shape and contact regular mesh sequence'' refers to \cite[Def.\ 1.38]{PietroE12} and requires that each mesh in the sequence admits a simplicial submesh without hanging nodes or deteriorating simplices.}

This can be further adapted to suit our purposes by observing that
\begin{equation}\label{EQ:discrete_poincare}
 \| v_h \|_{L^2(D)} \le \sigma \| v_h \|_\textup{dG} \le \underbrace{ \frac{\sigma}{\min\{\sqrt{a_{\min}(\bsy)},\sqrt{\eta(\vec y)}\}} }_{ = C^{V_h}_\text{Poin}(\vec y) } \| v_h \|_{V_h} \qquad \forall v_h \in V_h.
\end{equation}

The Poincar\'e constant $C^{V_h}_\text{Poin}(\vec y)$ can be chosen independently of $\vec y$ in the uniform and affine setting (there are uniform upper and lower bounds for the equivalence constants between norms), which is not possible (by our proof techniques---the equivalence constants between norms depend on $\vec y$) in the lognormal setting. \corr{Together with assumptions (U1)--(U2) and (L1)--(L2), the following Lemmas describe how $\eta(\bsy)$ should be chosen in order to ensure the integrability of $C_{{\rm Poin}}^{V_h}(\bsy)$.} 

\begin{lemma}[{\cite[Lem.\ 1.46]{PietroE12}}]\label{lemma:trace}
Let $\mathcal T_{h}$ be a member of a shape- and contact-regular mesh sequence. Then, for all $v_h\in V_h$, all $T\in\mathcal T_h$, all $F\subset\partial T$, \corr{and $h_T = \operatorname{diam}(T)$,}
$$
h_T^{1/2}\|v_h\|_{L^2(F)}\leq C_{\rm tr}\|v_h\|_{L^2(T)},
$$
where $C_{\rm tr}$ \corr{does not depend on $h$}.
\end{lemma}

This allows for a stability estimate of the approximate $u_h$.

\begin{lemma}\label{LEM:dg_stability}
 Assume that there is $\tau > 1$ (independent of $\vec y$ and $h$) such that
 \begin{equation}\label{EQ:penalty_condition}
  \eta(\vec y) \ge \tau \frac{a^2_{\max}(\vec y) C^2_{\textup{tr}} N_\partial (\theta - 1)^2}{4 a_{\min}(\vec y)},
 \end{equation}
 where $C_{\textup{tr}}$ denotes the constant of Lemma~\ref{lemma:trace}, $N_\partial$ is the maximum amount of faces of a mesh element (for simplicial meshes, $N_\partial = d+1$).
 There is $\alpha > 0$ (independent of $h$) such that
 \begin{multline*}%
  \alpha \| u_h \|^2_{V_h} \overset{\circlearound{A}}{\le} \dgBil(\vec y; u_h(\cdot, \vec y), u_h(\cdot, \vec y)) = F(u_h(\cdot, \vec y)) \\
  \le \| f \|_{L^2(D)} \| u_h \|_{L^2(D)} \le C^{V_h}_\textup{Poin}(\vec y) \| f \|_{L^2(D)} \| u_h \|_{V_h} ,
 \end{multline*}
 which gives the stability estimate
 \begin{equation}\label{EQ:dg_stability}
  \| u_h \|_{V_h} \le \frac{C^{V_h}_\textup{Poin}(\vec y)}{\alpha} \| f \|_{L^2(D)} \overset{\circlearound{B}}= \frac{\sigma}{\alpha \sqrt{a_{\min}(\bsy)}} \| f \|_{L^2(D)},
 \end{equation}
 where \circlearound{B} holds if $\eta(\vec y) \ge a_{\min}(\bsy)$. Additionally, for the NIPDG method ($\theta = 1$), inequality \circlearound{A} holds with equality and $\alpha = 1$ independent of $\eta(\vec y) > 0$.
\end{lemma}
\begin{proof}
 This is a direct consequence of \cite[Lem.\ 7.30]{KnabnerA21} except for identity \circlearound{B}. However, \circlearound{B} follows directly from combining \eqref{EQ:discrete_poincare} and $\eta (\vec y) \ge a_{\min}(\bsy)$.
\end{proof}

Additionally, we define the norm 
\begin{equation*}
 \| v_h \|^2_{V^\ast_h} = \| v_h \|^2_{V_h} + \sum_{\face \in \setFace} \frac{h_\face}{\eta(\vec y)} \| \avg{{a(\cdot, \vec y)} \nabla v_h} \|^2_{L^2(\face)}.
\end{equation*}

\begin{lemma}\label{LEM:norm_equiv}
 We have $\| v_h \|_{V_h} \le \| v_h \|_{V^\ast_h}$. Moreover, if there is $\tilde\tau > 0$ (independent of $\vec y$ and $h$) such that
 \begin{equation}\label{EQ:penalty_equiv}
  \eta(\vec y) \ge \tilde\tau \frac{a_{\max}(\vec y)^2}{a_{\min}(\vec y)} \qquad \text{ for all } \vec y,
 \end{equation}
 we have $C_\textup{equiv} > 0$ independent of $\vec y$ and $h$ such that
 \begin{equation*}%
  \| v_h \|_{V^\ast_h} \le C_\textup{equiv} \| v_h \|_{V_h} \qquad \text{ for all } v_h \in V_h.
 \end{equation*}
\end{lemma}
\begin{proof}
 The first inequality $\| v_h \|_{V_h} \le \| v_h \|_{V^\ast_h}$ is obvious. For the second inequality, we observe that we need to bound the additional term of the definition of $\|\cdot\|_{V^\ast_h}$ using
 \begin{align*}
  \sum_{\face \in \setFace} \frac{h_\face}{\eta(\vec y)} \| \avg{{a(\cdot, \vec y)} \nabla v_h} \|^2_{L^2(\face)} & \le \sum_{\face \in \setFace} \frac{a_{\max}(\vec y)^2 h_\face}{\eta(\vec y)} \| \avg{\nabla v_h} \|^2_{L^2(\face)} \\
  & \le \sum_{\face \in \setFace} \frac{a_{\max}(\vec y)^2 C^2_\textup{tr}}{\eta(\vec y)} \sum_{\elem \in \setElem\colon \partial \elem\supset\face} \| \nabla v_h \|^2_{L^2(\elem)} \\
  & \le \sum_{\elem \in \setElem} \frac{a_{\max}(\vec y)^2 C^2_\textup{tr} N_\partial}{\eta(\vec y)} \| \nabla v_h \|^2_{L^2(\elem)} \\
   & \le \sum_{\elem \in \setElem} \frac{a_{\max}(\vec y)^2 C^2_\textup{tr} N_\partial}{a_{\min}(\vec y) \eta(\vec y)} \| \sqrt{a(\cdot,\vec y)} \nabla v_h \|^2_{L^2(\elem)}.
 \end{align*}
 Here, the second inequality is the discrete trace inequality \corr{Lemma \ref{lemma:trace}, and uses that $h_F = \operatorname{diam}(F) \le \operatorname{diam}(T) = h_T$ for $F \subset T$}. The constant becomes independent of $\vec y$ if \eqref{EQ:penalty_equiv} holds.
\end{proof}

\emph{Remark.} In the lognormal setting, the condition of Lemma~\ref{LEM:dg_stability} can be satisfied by using the NIPG method corresponding to $\theta=1$. Lemma~\ref{LEM:norm_equiv} suggests that we can define
$$
\eta(\bsy):=\frac{a_{\max}(\bsy)^2}{a_{\min}(\bsy)}\quad\text{for all}~\bsy.
$$
\corr{Analogously, in the affine and uniform case, we can choose
$$
\eta(\bsy):= \eta = \frac{a_{\max}^2}{a_{\min}}\quad\text{for all}~\bsy.
$$
}Later on, we will choose $\eta(\bsy)$ precisely in this way.
\corr{Although the choice of $\eta(\vec y)$ is more subtle for the SIPG and IIPG methods, we experienced that too large $\eta(\vec y)$ do not have a visual effect on the results in Section \ref{SEC:numerics}, while too small $\eta(\vec y)$ can lead to wrong approximations. In principle, increasing $\eta$ adds interelement diffusion to the method. Thus, enlarging $\eta$ stabilizes the method, while its scaling ensures that convergence rates in $h$ do not deteriorate. Too large $\eta$ values have an adversarial effect on the constants of DG methods, but as long as integrability (concerning $\boldsymbol y$) is ensured, the QMC methods will work.}

\subsection{Bounds for derivatives with respect to the random variable}
\corr{The ultimate goal of our analysis is to derive regularity bounds in the flavor of \eqref{EQ:reg_bound_fe_aff} and \eqref{EQ:reg_bound_fe_log} for the diffusion problem that is discretized with an IPDG method instead of continuous finite elements. To this end, we start with some preliminary considerations before we prove Theorem \ref{TH:stability_derivative} stating that
\begin{multline*}
  \| \partial_{\bsy}^{\bsnu} u_h(\cdot, \vec y) \|_{V_h} \le \underbrace{ C^2_\textup{equiv} C'_\textup{DG} }_{ = C_\textup{DG} } \sum_{\mathbf 0\neq \vecm \le \bsnu} \begin{pmatrix} \bsnu \\ \vecm \end{pmatrix} \\
  \times\left[ \left\| \frac{\partial_{\bsy}^{\vecm} a(\cdot, \vec y)}{a(\cdot, \vec y)} \right\|_{L^\infty(D)} + \left| \frac{\partial_{\bsy}^{\bsm} \eta(\vec y)}{\eta(\vec y)} \right| \right] \| \partial_{\bsy}^{\bsnu-\bsm} u_h(\cdot, \vec y) \|_{V_h}
\end{multline*}
under certain conditions. Notably, the derivative $\partial_{\bsy}^{\bsnu} u_h(\cdot, \vec y)$ is an element of $V_h$. This can be seen by similar arguments as conducted in \cite{cds10} and is a well-known observation when approximating time-dependent equations by DG methods, where the derivative concerns a time unknown instead of a random variable, see \cite{CockburnS98} and \cite[Sect.\ 3.1]{PietroE12}.}

To derive \corr{this} regularity bound for the IPDG methods, we first observe that $\partial_{\bsy}^{\bsnu} \mathcal L(v_h) = 0$ \corr{(see \eqref{EQ:rhs})}, and that $a(\vec x, \vec y)$ is continuous with respect to $\vec x$, since $a_0$ and all $\psi_j$ have been assumed to be Lipschitz continuous. Thus,
\begin{align}
 &\dgBil(\vec y; u_h(\cdot, \vec y), v_h)\notag\\
 & = \sum_{\elem \in \setElem} \int_\elem a(\vec x, \vec y) \nabla u_h(\vec x, \vec y) \cdot \nabla v_h(\vec x) \dx\label{eq:star} \\
 & + \sum_{\elem \in \setElem} \int_{\partial \elem} \bigg[ \theta a(\vec x, \vec y) u_h(\vec x, \vec y) \normal_\elem(\vec x) \cdot \avg{\nabla v_h(\vec x)}\notag \\
 & - \frac{1}{2} a(\vec x, \vec y) \nabla u_h(\vec x, \vec y) \cdot \jump{v_h(\vec x)} + \frac{\eta(\vec y)}{h_\face} u_h(\vec x, \vec y) \normal_\elem(\vec x) \cdot \jump{v_h(\corr{\bsx})} \bigg] \ds,\notag
\end{align}
and we can deduce by $B_h(\bsy;u_h(\cdot,\bsy),v_h)=\mathcal L(v_h)$ and $\partial_{\bsy}^{\bsnu}\mathcal L(v_h)=0$ that
\begin{align*}
 0 =~ & \partial_{\bsy}^{\bsnu} \dgBil(\vec y; u_h(\cdot, \vec y), v_h) \\
 \overset{\eqref{eq:star}}{=}& \sum_{\elem \in \setElem} \int_\elem \partial_{\bsy}^{\bsnu} [ a(\vec x, \vec y) \nabla u_h(\vec x, \vec y) ] \cdot \nabla v_h(\vec x) \dx \\
 & + \sum_{\elem \in \setElem} \int_{\partial \elem} \bigg[ \theta \partial_{\bsy}^{\bsnu} [ a(\vec x, \vec y) u_h(\vec x, \vec y) ] \normal_\elem(\vec x) \cdot \avg{\nabla v_h(\vec x)} \\
 & \qquad - \frac{1}{2} \partial_{\bsy}^{\bsnu} [ a(\vec x, \vec y) \nabla u_h(\vec x, \vec y) ] \cdot \jump{v_h(\vec x)} \\
 & \qquad + h^{-1}_\face \partial_{\bsy}^{\bsnu} [ \eta(\vec y) u_h(\vec x, \vec y) ] \normal_\elem(\vec x) \cdot \jump{v_h(\corr{\bsx})} \bigg] \ds \\
 =~ & \sum_{\elem \in \setElem} \int_\elem \sum_{\vecm \le \bsnu} \begin{pmatrix} \bsnu \\ \vecm \end{pmatrix} \partial_{\bsy}^{\vecm} a(\vec x, \vec y) \nabla \partial_{\bsy}^{\bsnu-\vecm} u_h(\vec x, \vec y) \cdot \nabla v_h(\vec x) \dx \\
 & + \sum_{\elem \in \setElem} \int_{\partial \elem} \bigg[ \theta \sum_{\vecm \le \bsnu} \begin{pmatrix} \bsnu \\ \vecm \end{pmatrix} \partial_{\bsy}^{\vecm} a(\vec x, \vec y) \partial_{\bsy}^{\bsnu-\vecm} u_h(\vec x, \vec y) \normal_\elem(\vec x) \cdot \avg{\nabla v_h(\vec x)} \\
 & \qquad - \frac{1}{2} \sum_{\vecm \le \bsnu} \begin{pmatrix} \bsnu \\ \vecm \end{pmatrix} \partial_{\bsy}^{\vecm} a(\vec x, \vec y) \nabla \partial_{\bsy}^{\bsnu-\vecm} u_h(\vec x, \vec y) \cdot \jump{v_h(\vec x)} \\
 & \qquad + h^{-1}_\face \sum_{\vecm \le \bsnu} \begin{pmatrix} \bsnu \\ \vecm \end{pmatrix} \partial_{\bsy}^{\bsm} \eta(\vec y) \partial_{\bsy}^{\bsnu-\bsm} u_h(\vec x, \vec y) \normal_\elem(\vec x) \cdot \jump{v_h(\corr{\bsx})} \bigg] \ds \\
 =~ & \sum_{\elem \in \setElem} \int_\elem \sum_{\vecm \le \bsnu} \begin{pmatrix} \bsnu \\ \vecm \end{pmatrix} \partial_{\bsy}^{\vecm} a(\vec x, \vec y) \nabla \partial_{\bsy}^{\bsnu-\vecm} u_h(\vec x, \vec y) \cdot \nabla v_h(\vec x) \dx \\
 & + \sum_{\face\in \setFace} \int_{\face} \bigg[ \theta \sum_{\vecm \le \bsnu} \begin{pmatrix} \bsnu \\ \vecm \end{pmatrix} \partial_{\bsy}^{\vecm} a(\vec x, \vec y) \jump{ \partial_{\bsy}^{\bsnu-\vecm} u_h(\vec x, \vec y) } \cdot \avg{\nabla v_h(\vec x)} \\
 & \qquad - \sum_{\vecm \le \bsnu} \begin{pmatrix} \bsnu \\ \vecm \end{pmatrix} \partial_{\bsy}^{\vecm} a(\vec x, \vec y) \avg{ \nabla \partial_{\bsy}^{\bsnu-\vecm} u_h(\vec x, \vec y) } \cdot \jump{v_h(\vec x)} \\
 & \qquad + h^{-1}_\face \sum_{\vecm \le \bsnu} \begin{pmatrix} \bsnu \\ \vecm \end{pmatrix} \partial_{\bsy}^{\bsm} \eta(\vec y) \jump{ \partial_{\bsy}^{\bsnu-\bsm} u_h(\vec x, \vec y) } \cdot \jump{v_h(\corr{\bsx})} \bigg] \ds,
\end{align*}
where the third equality is a consequence of the Leibniz product rule and the fourth equality follows from the definitions of $\avg{\cdot}$ and $\jump{\cdot}$ analogously to~\eqref{eq:refertome}.

Next, we set $v_h = \partial_{\bsy}^{\bsnu} u_h(\cdot, \vec y)$ and do some algebraic manipulations to obtain
\begin{align*}
 &\| \partial_{\bsy}^{\bsnu} u_h(\cdot, \vec y) \|^2_{V_h}\\
 &=  - \sum_{\elem \in \setElem} \int_\elem \sum_{\mathbf 0\neq \vecm \le \bsnu} \!\begin{pmatrix} \bsnu \\ \vecm \end{pmatrix}\! \partial_{\bsy}^{\vecm} a(\vec x, \vec y) \nabla \partial_{\bsy}^{\bsnu-\vecm} u_h(\vec x, \vec y) \cdot \nabla \partial_{\bsy}^{\bsnu} u_h(\vec x, \vec y) \dx \\
 & \qquad - \sum_{\face\in \setFace} \int_{\face} \bigg[ \theta \sum_{\vecm \le \bsnu} \begin{pmatrix} \bsnu \\ \vecm \end{pmatrix} \partial_{\bsy}^{\vecm} a(\vec x, \vec y) \jump{ \partial_{\bsy}^{\bsnu-\vecm} u_h(\vec x, \vec y) } \cdot \avg{\nabla \partial_{\bsy}^{\bsnu} u_h(\vec x, \vec y)} \\
 & \qquad - \sum_{\vecm \le \bsnu} \begin{pmatrix} \bsnu \\ \vecm \end{pmatrix} \partial_{\bsy}^{\vecm} a(\vec x, \vec y) \avg{ \nabla \partial_{\bsy}^{\bsnu-\vecm} u_h(\vec x, \vec y) } \cdot \jump{\partial_{\bsy}^{\bsnu} u_h(\vec x, \vec y)} \\
 & \qquad + h^{-1}_\face \sum_{\mathbf 0\neq \vecm \le \bsnu} \begin{pmatrix} \bsnu \\ \vecm \end{pmatrix} \partial_{\bsy}^{\bsm} \eta(\vec y) \jump{ \partial_{\bsy}^{\bsnu-\bsm} u_h(\vec x, \vec y) } \cdot \jump{\partial_{\bsy}^{\bsnu} u_h(\vec x, \vec y)} \bigg] \ds \\
 &=  - \sum_{\elem \in \setElem} \int_\elem \sum_{\mathbf 0\neq \vecm \le \bsnu} \begin{pmatrix} \bsnu \\ \vecm \end{pmatrix} \partial_{\bsy}^{\vecm} a(\vec x, \vec y) \nabla \partial_{\bsy}^{\bsnu-\vecm} u_h(\vec x, \vec y) \cdot \nabla \partial_{\bsy}^{\bsnu} u_h(\vec x, \vec y) \dx \\
 & \qquad - \sum_{\face\in \setFace} \int_{\face} \bigg[ \theta \sum_{\mathbf 0\neq \vecm \le \bsnu} \begin{pmatrix} \bsnu \\ \vecm \end{pmatrix} \partial_{\bsy}^{\vecm} a(\vec x, \vec y) \jump{ \partial_{\bsy}^{\bsnu-\vecm} u_h(\vec x, \vec y) } \cdot \avg{\nabla \partial_{\bsy}^{\bsnu} u_h(\vec x, \vec y)} \\
 &  \qquad - \sum_{\mathbf 0\neq \vecm \le \bsnu} \begin{pmatrix} \bsnu \\ \vecm \end{pmatrix} \partial_{\bsy}^{\vecm} a(\vec x, \vec y) \avg{ \nabla \partial_{\bsy}^{\bsnu-\vecm} u_h(\vec x, \vec y) } \cdot \jump{\partial_{\bsy}^{\bsnu} u_h(\vec x, \vec y)} \\
 & \qquad + h^{-1}_\face \sum_{\mathbf 0\neq \vecm \le \bsnu} \begin{pmatrix} \bsnu \\ \vecm \end{pmatrix} \partial_{\bsy}^{\bsm} \eta(\vec y) \jump{ \partial_{\bsy}^{\bsnu-\bsm} u_h(\vec x, \vec y) } \cdot \jump{\partial_{\bsy}^{\bsnu} u_h(\vec x, \vec y)} \bigg] \ds \\
 & \qquad - \underbrace{ \sum_{\face\in \setFace} \int_{\face} (\theta - 1) a(\vec x, \vec y) \jump{ \partial_{\bsy}^{\bsnu} u_h(\vec x, \vec y) } \cdot \avg{\nabla \partial_{\bsy}^{\bsnu} u_h(\vec x, \vec y)} \ds }_{ =: \Psi_h (\partial_{\bsy}^{\bsnu} u_h)}
\end{align*}
For the second equality we separated $\boldsymbol m=\mathbf 0$ for face integrals from the rest.

Interestingly, if $\theta=1$, the highest order terms ($\boldsymbol m=\mathbf 0$) with respect to the faces cancel and  $\Psi_h (\partial_{\bsy}^{\bsnu} u_h) = 0$ for all $\eta(\vec y) > 0$. In any case, $\Psi_h$ can be bounded by first using H\"older's inequality, second absorbing $a$ into the constants (twice), third using Lemma~\ref{lemma:trace}, fourth incorporating $a_{\min}$ into the volume norms, and finally using Young's inequality. To wit:
\begin{align*}%
 | \Psi_h (\partial_{\bsy}^{\bsnu} u_h) | & \le \sum_{\face\in \setFace} \sqrt{\frac{a_{\min}(\vec y) h_\face}{\corr 2 a_{\max}(\vec y) C^2_{\textup{tr}} N_\partial}} \| \sqrt{a(\cdot, \vec y)} \avg{\nabla \partial_{\bsy}^{\bsnu} u_h(\cdot, \vec y)} \|_{L^2(\face)} \notag\\ 
 & \qquad \times |\theta - 1| \sqrt{\frac{\corr 2 a_{\max}(\vec y) C^2_{\textup{tr}} N_\partial}{a_{\min}(\vec y) h_\face}} \| \jump{\sqrt{ a(\cdot, \vec y) } \partial_{\bsy}^{\bsnu} u_h(\cdot, \vec y) } \|_{L^2(F)} \notag\\
 & \le \sum_{\face\in \setFace} \sqrt{\frac{a_{\min}(\vec y) h_\face}{\corr 2 C^2_{\textup{tr}} N_\partial}} \| \avg{\nabla \partial_{\bsy}^{\bsnu} u_h(\cdot, \vec y)} \|_{L^2(\face)} \notag\\
 & \qquad \times \sqrt{\frac{\corr 2 a^2_{\max}(\vec y) C^2_{\textup{tr}} N_\partial (\theta - 1)^2}{a_{\min}(\vec y) h_\face}} \| \jump{\partial_{\bsy}^{\bsnu} u_h(\cdot, \vec y) } \|_{L^2(F)} \notag\\
 & \le \left[ \sum_{\face\in \setFace} \frac{a_{\min}(\vec y)}{\corr 2 N_\partial} \left(\sum_{\elem \in \setElem\colon \partial \elem\supset\face} \| \nabla \partial_{\bsy}^{\bsnu} u_h(\cdot, \vec y) \|_{L^2(\elem)}\right)^2 \right]^{1/2} \notag\\
 & \qquad \times \left[ \sum_{\face \in \setFace} \frac{\corr 2 a^2_{\max}(\vec y) C^2_{\textup{tr}} N_\partial (\theta - 1)^2}{a_{\min}(\vec y) h_\face} \| \jump{\partial_{\bsy}^{\bsnu} u_h(\cdot, \vec y) } \|^2_{L^2(F)} \right]^{1/2} \notag\\
 & \le \left[ \sum_{\elem\in \setElem} \| \sqrt{a(\cdot, \vec y)} \nabla \partial_{\bsy}^{\bsnu} u_h(\cdot, \vec y) \|^2_{L^2(\elem)} \right]^{1/2} \notag\\
 & \qquad \times \sqrt{\frac{\corr 2 a^2_{\max}(\vec y) C^2_{\textup{tr}} N_\partial (\theta - 1)^2}{a_{\min}(\vec y)}} \left[ \sum_{\face\in \setFace} \frac{1}{h_\face} \| \jump{\partial_{\bsy}^{\bsnu} u_h(\cdot, \vec y) } \|^2_{L^2(F)} \right]^{1/2} \\
 & \le \delta \sum_{\elem\in \setElem} \| \sqrt{a(\cdot, \vec y)} \nabla \partial_{\bsy}^{\bsnu} u_h(\cdot, \vec y) \|^2_{L^2(\elem)} \notag\\
 & \qquad + \frac{a^2_{\max}(\vec y) C^2_{\textup{tr}} N_\partial (\theta - 1)^2}{\corr 2 \delta a_{\min}(\vec y)} \sum_{\face\in \setFace} \frac{1}{h_\face} \| \jump{\partial_{\bsy}^{\bsnu} u_h(\cdot, \vec y) } \|^2_{L^2(F)} \notag
\end{align*}
and $\delta \in (0,1)$ can be chosen arbitrarily. \corr{Notably, the fourth inequality uses that
\begin{multline*}
 \sum_{\face\in \setFace} \frac{a_{\min}(\vec y)}{2 N_\partial} \left(\sum_{\elem \in \setElem\colon \partial \elem\supset\face} \| \nabla \partial_{\bsy}^{\bsnu} u_h(\cdot, \vec y) \|_{L^2(\elem)}\right)^2 \\
 \le \frac{1}{2 N_\partial} \sum_{\face\in \setFace} \left(\sum_{\elem \in \setElem\colon \partial \elem\supset\face} \| \sqrt{a(\cdot,\vec y)} \nabla \partial_{\bsy}^{\bsnu} u_h(\cdot, \vec y) \|_{L^2(\elem)}\right)^2,
\end{multline*}
and the fact that each element has at most $N_\partial$ faces. Hence, each element appears at most $N_\partial$ times on the right-hand side. Moreover, we have that
\begin{multline*}
 \left(\sum_{\elem \in \setElem\colon \partial \elem\supset\face} \| \sqrt{a(\cdot,\vec y)} \nabla \partial_{\bsy}^{\bsnu} u_h(\cdot, \vec y) \|_{L^2(\elem)}\right)^2 \\ \le 2 \sum_{\elem \in \setElem\colon \partial \elem\supset\face} \| \sqrt{a(\cdot,\vec y)} \nabla \partial_{\bsy}^{\bsnu} u_h(\cdot, \vec y) \|^2_{L^2(\elem)}
\end{multline*}
as each face has at most two neighbors.
} %
Thus, if \eqref{EQ:penalty_condition} holds, we can choose $\delta = \tfrac{2}{\tau+1}$, and $\Psi_h (\partial_{\bsy}^{\bsnu} u_h) < \| \partial_{\bsy}^{\bsnu} u_h(\cdot, \vec y) \|^2_{V_h}$ can be incorporated in the norm on the left-hand side producing a constant that is independent of $\vec y$:
\begin{align*}
 \| \partial_{\bsy}^{\bsnu} & u_h(\cdot, \vec y) \|^2_{V_h} \\
 \overset{\circlearound{A}}{\lesssim} & - \sum_{\elem \in \setElem} \int_\elem \sum_{\mathbf 0\neq \vecm \le \bsnu} \begin{pmatrix} \bsnu \\ \vecm \end{pmatrix} \partial_{\bsy}^{\vecm} a(\vec x, \vec y) \nabla \partial_{\bsy}^{\bsnu-\vecm} u_h(\vec x, \vec y) \cdot \nabla \partial_{\bsy}^{\bsnu} u_h(\vec x, \vec y) \dx \\
 & - \sum_{\face\in \setFace} \int_{\face} \bigg[ \theta \sum_{\mathbf 0\neq \vecm \le \bsnu} \begin{pmatrix} \bsnu \\ \vecm \end{pmatrix} \partial_{\bsy}^{\vecm} a(\vec x, \vec y) \jump{ \partial_{\bsy}^{\bsnu-\vecm} u_h(\vec x, \vec y) } \cdot \avg{\nabla \partial_{\bsy}^{\bsnu} u_h(\vec x, \vec y)} \\
 & \qquad - \sum_{\mathbf 0\neq \vecm \le \bsnu} \begin{pmatrix} \bsnu \\ \vecm \end{pmatrix} \partial_{\bsy}^{\vecm} a(\vec x, \vec y) \avg{ \nabla \partial_{\bsy}^{\bsnu-\vecm} u_h(\vec x, \vec y) } \cdot \jump{\partial_{\bsy}^{\bsnu} u_h(\vec x, \vec y)} \\
 & \qquad + h^{-1}_\face \sum_{\mathbf 0\neq\vecm \le \bsnu} \begin{pmatrix} \bsnu \\ \vecm \end{pmatrix} \partial_{\bsy}^{\bsm} \eta(\vec y) \jump{ \partial_{\bsy}^{\bsnu-\bsm} u_h(\vec x, \vec y) } \cdot \jump{\partial_{\bsy}^{\bsnu} u_h(\vec x, \vec y)} \bigg] \ds \\
 = & - \!\sum_{\mathbf 0\neq \vecm \le \bsnu} \!\!\begin{pmatrix} \bsnu \\ \vecm \end{pmatrix}\!\! \Bigg[ \!\sum_{\elem \in \setElem}\! \int_\elem\!  \frac{\partial_{\bsy}^{\vecm} a(\vec x, \vec y)}{a(\vec x, \vec y)} a(\vec x, \vec y) \nabla \partial_{\bsy}^{\bsnu-\vecm} u_h(\vec x, \vec y) \!\cdot\! \nabla \partial_{\bsy}^{\bsnu} u_h(\vec x, \vec y) \dx \\
 & + \sum_{\face\in \setFace} \int_{\face} \frac{\partial_{\bsy}^{\vecm} a(\vec x, \vec y)}{a(\vec x, \vec y)} \bigg[ \theta \jump{ \partial_{\bsy}^{\bsnu-\vecm} u_h(\vec x, \vec y) } \cdot \avg{a(\vec x, \vec y) \nabla \partial_{\bsy}^{\bsnu} u_h(\vec x, \vec y)} \\
 & \qquad - \avg{a(\vec x, \vec y) \nabla \partial_{\bsy}^{\bsnu-\vecm} u_h(\vec x, \vec y) } \cdot \jump{\partial_{\bsy}^{\bsnu} u_h(\vec x, \vec y)}  \bigg]\,{\rm d}\sigma \\
 &  + \sum_{F\in\mathcal F}\int_F\frac{\partial_{\bsy}^{\bsm} \eta(\vec y)}{\eta(\vec y)}  \frac{\eta(\vec y)}{h_\face} \jump{ \partial_{\bsy}^{\bsnu-\bsm} u_h(\vec x, \vec y) } \cdot \jump{\partial_{\bsy}^{\bsnu} u_h(\vec x, \vec y)} \ds \Bigg] \quad = (\bigstar),
\end{align*}
where `$\lesssim$' is to be interpreted as smaller than or equal to up to a positive, multiplicative constant, which does not depend on $h$, $\bsnu$, or $\vec y$. Importantly, for the non-symmetric IP method ($\theta = 1$), the symbol `$\lesssim$' at \circlearound{A} can be replaced by `$=$' independently of $\tau$. Otherwise $\tau\searrow 1$ yields an unstable method.

Our next task consists in estimating $| (\bigstar) |$. H\"older's inequality yields
\begin{subequations}
\begin{align}
 | (\bigstar) | \le & \sum_{\mathbf 0\neq \vecm \le \bsnu} \begin{pmatrix} \bsnu \\ \vecm \end{pmatrix} \left\| \frac{\partial_{\bsy}^{\vecm} a(\cdot, \vec y)}{a(\cdot, \vec y)} \right\|_{L^\infty(D)} \notag \\
 & \times \Bigg[ \sum_{\elem \in \setElem} \| \sqrt{a(\cdot, \vec y)} \nabla \partial_{\bsy}^{\bsnu-\vecm} u_h(\cdot, \vec y) \|_{L^2(\elem)} \| \sqrt{a(\cdot, \vec y)} \nabla \partial_{\bsy}^{\bsnu} u_h(\cdot, \vec y) \|_{L^2(D)} \label{EQ:elem_term}\\
 &\quad + \sum_{\face\in \setFace} \int_{\face} \left| \theta \jump{ \partial_{\bsy}^{\bsnu-\vecm} u_h(\vec x, \vec y) } \cdot \avg{a(\vec x, \vec y) \nabla \partial_{\bsy}^{\bsnu} u_h(\vec x, \vec y)} \right| \ds \label{EQ:face_term1}\\
 &\quad + \sum_{\face\in \setFace} \int_{\face} \left| \avg{a(\vec x, \vec y) \nabla \partial_{\bsy}^{\bsnu-\vecm} u_h(\vec x, \vec y) } \cdot \jump{\partial_{\bsy}^{\bsnu} u_h(\vec x, \vec y)} \right| \ds \Bigg]\label{EQ:face_term2} \\
 & + \sum_{\mathbf 0\neq \vecm \le \bsnu} \begin{pmatrix} \bsnu \\ \vecm \end{pmatrix} \left| \frac{\partial_{\bsy}^{\bsm} \eta(\vec y)}{\eta(\vec y)} \right| \notag \\
 & \quad \times \sum_{\face\in \setFace} \int_{\face} \frac{\eta(\vec y)}{h_\face} \jump{ \partial_{\bsy}^{\bsnu-\bsm} u_h(\vec x, \vec y) } \cdot \jump{\partial_{\bsy}^{\bsnu} u_h(\vec x, \vec y)} \ds. \label{EQ:jump_der}
\end{align}
\end{subequations}
The former three lines can be bounded separately using $|\theta| \le 1$: 
\begin{align*}
 |\eqref{EQ:elem_term}|  & \le \left[ \sum_{\elem \in \setElem} \| \sqrt{a(\cdot, \vec y)} \nabla \partial_{\bsy}^{\bsnu} u_h(\cdot, \vec y) \|^2_{L^2(D)}\right]^{\tfrac 12} \\
 & \qquad \times \left[\sum_{\elem \in \setElem} \| \sqrt{a(\cdot, \vec y)} \nabla \partial_{\bsy}^{\bsnu-\vecm} u_h(\cdot, \vec y) \|^2_{L^2(\elem)}\right]^{\tfrac 12}, \\
 |\eqref{EQ:face_term1}| & \le \left[ \sum_{\face\in \setFace} \frac{\eta(\vec y)}{h_\face} \| \jump{\partial_{\bsy}^{\bsnu-\bsm} u_h(\cdot, \vec y) } \|^2_{L^2(F)} \right]^{\tfrac 12} \\
 & \qquad \times \left[ \sum_{\face\in \setFace} \frac{h_\face}{\eta(\vec y)} \| \avg{ {a(\cdot, \vec y)} \nabla \partial_{\bsy}^{\bsnu} u_h(\cdot, \vec y) } \|^2_{L^2(\face)} \right]^{\tfrac 12}, \\
 |\eqref{EQ:face_term2}| & \le \left[ \sum_{\face\in \setFace} \frac{\eta(\vec y)}{h_\face} \| \jump{\partial_{\bsy}^{\bsnu} u_h(\cdot, \vec y) } \|^2_{L^2(F)} \right]^{\tfrac 12} \\
 & \qquad \times \left[ \sum_{\face\in \setFace} \frac{h_\face}{\eta(\vec y)} \| \avg{ {a(\cdot, \vec y)} \nabla \partial_{\bsy}^{\bsnu-\bsm} u_h(\cdot, \vec y) } \|^2_{L^2(\face)} \right]^{\tfrac 12}
\end{align*}
and the last line can be estimated via
\begin{equation*}
 | \eqref{EQ:jump_der} | \!\le\! \left[ \sum_{\face\in \setFace}\! \frac{\eta(\vec y)}{h_\face} \| \jump{\partial_{\bsy}^{\bsnu-\bsm} u_h(\cdot, \vec y) } \|^2_{L^2(F)} \!\right]^{\tfrac 12} \! \left[\! \sum_{\face\in \setFace} \frac{\eta(\vec y)}{h_\face} \| \jump{\partial_{\bsy}^{\bsnu} u_h(\cdot, \vec y) } \|^2_{L^2(F)} \!\right]^{\tfrac 12}\!\!.
\end{equation*}
As a consequence, we obtain by the Cauchy--Schwarz inequality that
\begin{multline*}
 | (\bigstar) | \le \sum_{\mathbf 0\neq \vecm \le \bsnu} \begin{pmatrix} \bsnu \\ \vecm \end{pmatrix} \left[ \left\| \frac{\partial_{\bsy}^{\vecm} a(\cdot, \vec y)}{a(\cdot, \vec y)} \right\|_{L^\infty(D)} + \left| \frac{\partial_{\bsy}^{\bsm} \eta(\vec y)}{\eta(\vec y)} \right| \right] \\ \times\| \partial_{\bsy}^{\bsnu} u_h(\cdot, \vec y) \|_{V^\ast_h} \| \partial_{\bsy}^{\bsnu-\bsm} u_h(\cdot, \vec y) \|_{V^\ast_h}.
\end{multline*}

Thus, considering Lemma \ref{LEM:norm_equiv} and division by $\| \partial_{\bsy}^{\bsnu} u_h(\cdot, \vec y) \|_{V_h}$ yield the following.

\begin{theorem}\label{TH:stability_derivative}
 If \eqref{EQ:penalty_condition} and \eqref{EQ:penalty_equiv} hold, there are constants $C'_\textup{DG}, C_\textup{DG} > 0$ (independent of $\vec y$ and $h$) such that
 \begin{multline}\label{EQ:dg_induc_step}
  \| \partial_{\bsy}^{\bsnu} u_h(\cdot, \vec y) \|_{V_h} \le \underbrace{ C^2_\textup{equiv} C'_\textup{DG} }_{ = C_\textup{DG} } \sum_{\mathbf 0\neq \vecm \le \bsnu} \begin{pmatrix} \bsnu \\ \vecm \end{pmatrix} \\
  \times\left[ \left\| \frac{\partial_{\bsy}^{\vecm} a(\cdot, \vec y)}{a(\cdot, \vec y)} \right\|_{L^\infty(D)} + \left| \frac{\partial_{\bsy}^{\bsm} \eta(\vec y)}{\eta(\vec y)} \right| \right] \| \partial_{\bsy}^{\bsnu-\bsm} u_h(\cdot, \vec y) \|_{V_h}.
 \end{multline}
 For NIPDG ($\theta = 1$), one may choose $C'_\textup{DG} = 1$.
\end{theorem}

\subsection{A brief note on \corr{error estimates}}\label{sec:brief}
The error estimate does not significantly deviate from standard DG error estimates. Thus, we keep it short and refer to the already mentioned references \cite{PietroE12,KnabnerA21} for the details. We start introducing the broken $H^2$ space
\begin{equation*}
 H^2(\setElem) = \{ v \in L^2(D) \colon v|_\elem \in H^2(\elem) \}
\end{equation*}
and observe the following.
\begin{corollary}
 Under the assumptions of Lemma \ref{LEM:norm_equiv} there is a constant $M$ (independent of $\vec y$ and $h$) such that for all $w_h \in H^2(\setElem)$ and all $v_h \in V_h$, we have
 \begin{equation*}
  B_h(w_h, v_h) \le M \| w_h \|_{V_h^\ast} \| v_h \|_{V_h}.
 \end{equation*}
\end{corollary}
This follows directly from Cauchy--Schwarz inequality and Lemma \ref{LEM:norm_equiv}. Thus, Strang's lemma \cite[Rem.\ 4.9]{KnabnerA21} implies that for any $w_h \in V_h$
\begin{equation*}
 \| u - u_h \|_{V^\ast_h} \le \left( 1 + \frac{MC_\textup{equiv}}{\alpha} \right) \| u - w_h \|_{V^\ast_h}.
\end{equation*}

We choose $w_h$ to be the local $L^2$ projection $\Pi u$ of $u$, and will need to estimate the respective terms of $\| u - \Pi u \|_{V^\ast_h}$. Exploiting the standard scaling arguments of DG analysis, we arrive at the following.
\begin{theorem}
 Let $u(\cdot, \vec y) \in H^{k+1}(D)$ and the assumptions of Lemmas \ref{LEM:dg_stability} and \ref{LEM:norm_equiv} hold, then we have
 \begin{multline*}
  \frac{1}{C^{V_h}_\textup{Poin}(\bsy)} \| u(\cdot, \vec y) - u_h(\cdot, \vec y) \|_{L^2(D)} \le \| u(\cdot, \vec y) - u_h(\cdot, \vec y) \|_{V^\ast_h} \\
  \lesssim \max\left\{ \sqrt{\eta(\vec y)}, {a_{\max}(\vec y)}, \sqrt{\tfrac{1}{\eta(\vec y)}}\right\} h^k |u(\cdot,\vec y)|_{H^{k+1}(D)}.
 \end{multline*}
\end{theorem}
Notably, the convergence order for SIPG can be increased by one using a duality argument as in \cite[Sect.\ 4.2.4]{PietroE12} if we have elliptic regularity. Importantly, the FE error in the QMC setting will be integrated with respect to $\vec y$. That is why, we additionally assume that $|u(\cdot,\vec y)|_{H^{k+1}(D)}$ is integrable with respect to $\vec y$. By standard elliptic regularity theory, this holds for example when $D$ is a bounded, convex domain, the diffusion coefficient $a$ is Lipschitz continuous, and the source term $f\in L^2(D)$. For $k=1$, the integrability of the seminorm is shown in~\cite{teckentrup}.

\section{Parametric regularity analysis}\label{sec:parametric}
\subsection{Uniform and affine setting}
In this setting, the regularity analysis is completely analogous to the conforming FE setting. Specifically, the choice of $\eta(\bsy)$ is not as critical as in the lognormal case to be discussed later on: $\eta(\bsy)$ can be chosen independently of $\bsy$ and large enough to satisfy Lemmas~\ref{LEM:dg_stability} and~\ref{LEM:norm_equiv}.
\begin{lemma}[Regularity bound in the uniform and affine setting]\label{lemma:affineregularity}
Let assumptions {\rm (U1)}--{\rm (U5)} hold with $\eta$ sufficiently large (independently of $\bsy$). Let $\bsb=(b_j)_{j\geq 1}$ with $b_j:=\frac{C_\textup{DG} \|\psi_j\|_{L^\infty(D)}}{\alpha a_{\min}}$. Then for all $\bsnu\in\mathscr F$ and $\bsy\in U$, it holds that
$$
\|\partial_{\bsy}^{\bsnu}u_h(\cdot,\bsy)\|_{V_h}\leq |\bsnu|!\bsb^{\bsnu}\frac{C_{\rm Poin}^{V_h}}{\alpha}\|f\|_{L^2(D)}.
$$
\end{lemma}
\begin{proof}
This is an immediate consequence of~\cite[Lem.~9.1]{kuonuyenssurvey} (in fact, it already appears in~\cite{cds10}), but since the proof is short, we present it for completeness. The proof is carried out by induction with respect to the order of the multi-indices $\bsnu\in\mathscr F$. In the affine setting, we make use of the fact that
$$
\partial_{\bsy}^{\bsnu}a(\bsx,\bsy)=\begin{cases}
a(\bsx,\bsy)&\text{if}~\bsnu=\mathbf 0,\\
\psi_j(\bsx)&\text{if}~\bsnu=\boldsymbol e_j,~j\geq 1,\\
0&\text{otherwise},
\end{cases}
$$
where $\boldsymbol e_j$ denotes the $j^{\rm th}$ Euclidean standard unit vector.

The base of the induction is resolved by observing that the claim is equivalent to the {\em a priori} bound when $\bsnu=\mathbf 0$, see Lemma \ref{LEM:dg_stability}.

Let $\bsnu\in\mathscr F$ and suppose that the claim has already been proved for all multi-indices with order less than $|\bsnu|$. Then it is a consequence of Theorem~\ref{TH:stability_derivative} \corr{and the fact that $\partial^{\boldsymbol m}\eta(\boldsymbol y)=0$ for $|\boldsymbol m|\geq 1$} that
\begin{align*}
 \| \partial_{\bsy}^{\bsnu} u_h(\cdot, \vec y) \|_{V_h} &\leq C_\textup{DG} \sum_{\mathbf 0\neq \vecm \le \bsnu} \begin{pmatrix} \bsnu \\ \vecm \end{pmatrix} \bigg\| \frac{\partial_{\bsy}^{\vecm} a(\cdot, \vec y)}{a(\cdot,\vec y)} \bigg\|_{L^\infty(D)} \| \partial_{\bsy}^{\bsnu-\vecm} u_h(\cdot, \vec y) \|_{V_h}\\
 &\leq\sum_{j\in{\rm supp}(\bsnu)} \nu_jb_j |\bsnu-\boldsymbol e_j|!\bsb^{\bsnu-\boldsymbol e_j}\frac{C_{\rm Poin}^{V_h}}{\alpha}\|f\|_{L^2(D)}\\
  &=\bsb^{\bsnu}(|\bsnu|-1)!\bigg(\sum_{j\in{\rm supp}(\bsnu)} \nu_j\bigg)\frac{C_{\rm Poin}^{V_h}}{\alpha}\|f\|_{L^2(D)} \\
    &=\bsb^{\bsnu}|\bsnu|!\frac{C_{\rm Poin}^{V_h}}{\alpha}\|f\|_{L^2(D)},
\end{align*}
as desired.
\end{proof}

\subsection{Derivation of QMC error in the uniform and affine setting}\label{sec:unifqmc}

In what follows, we define the dimensionally-truncated solution by setting
$$
u_{s,h}(\cdot,\bsy):=u_h(\cdot,(y_1,\ldots,y_s,0,0,\ldots)),\quad \bsy\in [-\tfrac12,\tfrac12]^s.
$$

\begin{theorem}\label{thm:qmcweight}
Let assumptions {\rm (U1)}--{\rm (U5)} hold and suppose that $n=2^m$, $m\in\mathbb N$. There exists a generating vector constructed by the CBC algorithm such that
\begin{align*}
&\sqrt{\mathbb E_{\boldsymbol\Delta}\bigg\|\int_{\corr{[-1/2,1/2]^s}} u_{s,h}(\cdot,\bsy)\,{\rm d}\bsy-Q_{\rm ran}(u_{s,h})\bigg\|_{L^2(D)}^2}\\
&=\begin{cases}
\mathcal O(n^{-1/p+1/2})&\text{if}~p\in (2/3,1),\\
\mathcal O(n^{-1+\varepsilon})~\text{for arbitrary}~\varepsilon\in(0,1/2)&\text{if}~p\in (0,2/3],
\end{cases}
\end{align*}
where the implied coefficient is independent of $s$ in both cases, when the weights $(\gamma_{\setu})_{\setu\subseteq\{1:s\}}$ in Lemma~\ref{lemma:affineqmc} are chosen to be
$$
\gamma_{\setu}=\bigg(|\setu|!\prod_{j\in\setu}\frac{b_j}{\sqrt{\varrho(\lambda)}}\bigg)^{2/(1+\lambda)}\quad\text{for}~\setu\subseteq\{1:s\}
$$
and
$$
\lambda=\begin{cases}
\frac{p}{2-p}&\text{if}~p\in (2/3,1),\\
\frac{1}{2-2\varepsilon}~\text{for arbitrary}~\varepsilon\in(0,1/2)&\text{if}~p\in (0,2/3].
\end{cases}
$$
\end{theorem}
\begin{proof} In the uniform and affine setting, $C_{\rm Poin}^{V_h}$ is independent of $\bsy$ and so we can proceed with the usual derivation of QMC weights. Let the assumptions of Lemma~\ref{lemma:affineqmc} be in effect. Furthermore, suppose that assumptions (U1)--(U3) and (U5) hold. Then for almost all $\bsx\in D$, we have from Theorem~\ref{lemma:affineqmc} that
\begin{align*}
&\mathbb E_{\boldsymbol\Delta}\bigg|\int_{\corr{[-1/2,1/2]^s}} u_{s,h}(\bsx,\bsy)\,{\rm d}\bsy-Q_{\rm ran}(u_{s,h}(\bsx,\cdot))\bigg|^2\\
&\leq \bigg(\frac{2}{n}\bigg)^{1/\lambda}\frac{1}{R}\bigg(\sum_{\emptyset\neq\setu\subseteq\{1:s\}}\gamma_{\setu}^\lambda \varrho(\lambda)^{|\setu|}\bigg)^{1/\lambda} \\
& \qquad \times \sum_{\setu\subseteq\{1:s\}}\frac{1}{\gamma_{\setu}}\int_{[-1/2,1/2]^{|\setu|}}\bigg(\int_{[-1/2,1/2]^{s-|\setu|}}\frac{\partial^{|\setu|}}{\partial \bsy_{\setu}}u_{s,h}(\bsx,\bsy)\,{\rm d}\bsy_{-\setu}\bigg)^2\,{\rm d}\bsy_{\setu}\\
&\leq \!\bigg(\frac{2}{n}\bigg)^{1/\lambda}\!\frac{1}{R}\bigg(\sum_{\emptyset\neq\setu\subseteq\{1:s\}}\!\gamma_{\setu}^\lambda \varrho(\lambda)^{|\setu|}\!\bigg)^{1/\lambda}\!\!\sum_{\setu\subseteq\{1:s\}}\!\frac{1}{\gamma_{\setu}}\int_{[-1/2,1/2]^s}\!\bigg|\frac{\partial^{|\setu|}}{\partial \bsy_{\setu}}u_{s,h}(\bsx,\bsy)\bigg|^2\,\!{\rm d}\bsy.
\end{align*}
Integrating over $\bsx\in D$ on both sides and using Fubini's theorem, we obtain
\begin{align*}
&\mathbb E_{\boldsymbol\Delta}\bigg\|\int_{\corr{[-1/2,1/2]^s}} u_{s,h}(\cdot,\bsy)\,{\rm d}\bsy-Q_{\rm ran}(u_{s,h})\bigg\|_{L^2(D)}^2\\
&\leq \bigg(\frac{2}{n}\bigg)^{1/\lambda}\frac{1}{R}\bigg(\sum_{\emptyset\neq\setu\subseteq\{1:s\}}\gamma_{\setu}^\lambda \varrho(\lambda)^{|\setu|}\bigg)^{1/\lambda}\\
&\qquad \times\sum_{\setu\subseteq\{1:s\}}\frac{1}{\gamma_{\setu}}\int_{[-1/2,1/2]^s}\bigg\|\frac{\partial^{|\setu|}}{\partial \bsy_{\setu}}u_{s,h}(\cdot,\bsy)\bigg\|_{L^2(D)}^2\,{\rm d}\bsy\\
&\leq \bigg(\frac{2}{n}\bigg)^{1/\lambda}\frac{(C_{\rm Poin}^{V_h})^2}{R}\bigg(\sum_{\emptyset\neq\setu\subseteq\{1:s\}}\gamma_{\setu}^\lambda \varrho(\lambda)^{|\setu|}\bigg)^{1/\lambda} \\
& \qquad \times \sum_{\setu\subseteq\{1:s\}}\frac{1}{\gamma_{\setu}}\int_{[-1/2,1/2]^s}\bigg\|\frac{\partial^{|\setu|}}{\partial \bsy_{\setu}}u_{s,h}(\cdot,\bsy)\bigg\|_{V_h}^2\,{\rm d}\bsy,
\end{align*}
where we used the discrete Poincar\'e inequality \eqref{EQ:discrete_poincare}. (Recall that in the uniform and affine setting, the Poincar\'e constant can be bounded independently of $\bsy$.) We can now apply Lemma~\ref{lemma:affineregularity} to obtain
\begin{align*}
&\mathbb E_{\boldsymbol\Delta}\bigg\|\int_{\corr{[-1/2,1/2]^s}} u_{s,h}(\cdot,\bsy)\,{\rm d}\bsy-Q_{\rm ran}(u_{s,h})\bigg\|_{L^2(D)}^2\\
&\leq \bigg(\frac{2}{n}\bigg)^{1/\lambda}\frac{(C_{\rm Poin}^{V_h})^4\|f\|_{L^2(D)}^2}{\alpha^2 R}\bigg(\sum_{\emptyset\neq\setu\subseteq\{1:s\}}\gamma_{\setu}^\lambda \varrho(\lambda)^{|\setu|}\bigg)^{1/\lambda}\sum_{\setu\subseteq\{1:s\}}\frac{1}{\gamma_{\setu}}(|\setu|!)^2\prod_{j\in\setu}b_j^2.
\end{align*}
It is easy to check that the upper bound is minimized by choosing (cf., e.g.,~\cite[Lem.~6.2]{kss12})
\begin{align}
\gamma_{\setu}=\bigg(|\setu|!\prod_{j\in\setu}\frac{b_j}{\sqrt{\varrho(\lambda)}}\bigg)^{2/(1+\lambda)}\quad\text{for}~\setu\subseteq\{1:s\}.\label{eq:qmcweightsaffine}
\end{align}
By plugging these weights into the error bound, we obtain
\begin{align*}
&\sqrt{\mathbb E_{\boldsymbol\Delta}\bigg\|\int_{\corr{[-1/2,1/2]^s}} u_{s,h}(\cdot,\bsy)\,{\rm d}\bsy-Q_{\rm ran}(u_{s,h})\bigg\|_{L^2(D)}^2}\\
&\leq \bigg(\frac{2}{n}\bigg)^{1/(2\lambda)}\frac{(C_{\rm Poin}^{V_h})^2\|f\|_{L^2(D)}}{\alpha\sqrt R}\bigg(\underset{=:C(s,\lambda)}{\underbrace{\sum_{\setu\subseteq\{1:s\}}\bigg(|\setu|!\prod_{j\in\setu}\frac{b_j}{\sqrt{\varrho(\lambda)}}\bigg)^{\frac{2\lambda}{1+\lambda}} \varrho(\lambda)^{|\setu|}}}\bigg)^{\frac{\lambda+1}{2\lambda}}.
\end{align*}
The term $C(s,\lambda)$ can be bounded independently of $s$:
\begin{align*}
C(s,\lambda)&=\sum_{\ell=0}^s (\ell!)^{2\lambda/(1+\lambda)}\varrho(\lambda)^{\frac{1}{1+\lambda}\ell}\sum_{\substack{\setu\subseteq\{1:s\}\\ |\setu|=\ell}}\prod_{j\in\setu}b_j^{2\lambda/(1+\lambda)}\\
&\leq \sum_{\ell=0}^\infty (\ell!)^{2\lambda/(1+\lambda)}\varrho(\lambda)^{\frac{1}{1+\lambda}\ell}\frac{1}{\ell!}\bigg(\sum_{j=1}^\infty b_j^{2\lambda/(1+\lambda)}\bigg)^{\ell},
\end{align*}
where the series can be shown to converge by the d'Alembert ratio test provided that
$$
\frac{2\lambda}{1+\lambda}\geq p\quad\text{and}\quad \frac{1}{2}<\lambda\leq 1.
$$
We can ensure that both conditions are satisfied by choosing
\begin{equation}
\lambda=\begin{cases}
\frac{p}{2-p}&\text{if}~p\in (2/3,1),\\
\frac{1}{2-2\varepsilon}~\text{for arbitrary}~\varepsilon\in(0,1/2)&\text{if}~p\in (0,2/3].
\end{cases}\label{eq:lambdaaffine}
\end{equation}
We conclude that by choosing the weights~\eqref{eq:qmcweightsaffine} and the parameter~\eqref{eq:lambdaaffine}, we obtain the QMC convergence rate
\begin{align*}
&\sqrt{\mathbb E_{\boldsymbol\Delta}\bigg\|\int_{\corr{[-1/2,1/2]^s}} u_{s,h}(\bsx,\bsy)\,{\rm d}\bsy-Q_{\rm ran}(u_{s,h}(\bsx,\cdot))\bigg\|_{L^2(D)}^2}\\
&=\begin{cases}
\mathcal O(n^{-1/p+1/2})&\text{if}~p\in (2/3,1),\\
\mathcal O(n^{-1+\varepsilon})~\text{for arbitrary}~\varepsilon\in(0,1/2)&\text{if}~p\in (0,2/3],
\end{cases}
\end{align*}
where the implied coefficient is independent of $s$ in both cases. This concludes the proof.
\end{proof}

\corr{{\em Remark.} Let $\mathcal G\!:V_h\to \mathbb R$ be a bounded linear functional. The same proof technique can be applied {\em mutatis mutandis} to show that there exists a generating vector constructed by the CBC algorithm using the weights specified in Theorem~\ref{thm:qmcweight} such that
$$
\sqrt{\mathbb E_{\boldsymbol\Delta}\bigg|\int_{[-1/2,1/2]^s}\mathcal G(u_{s,h}(\cdot,\bsy))\,{\rm d}\bsy-Q_{\rm ran}(\mathcal G(u_{s,h}))\bigg|^2}=\mathcal O(n^{\max\{-1/p+1/2,-1+\varepsilon\}}),
$$
where the implied coefficient is independent of the dimension $s$ with arbitrary $\varepsilon\in(0,1/2)$.
}

We note that the QMC convergence rate is at best essentially linear, and it is always better than $\mathcal O(n^{-1/2})$.

\subsection{Lognormal setting}
In the lognormal setting, we set
\begin{align}
\eta(\boldsymbol{y}):=\frac{(\max_{\boldsymbol{x}\in\overline{D}}a_0(\boldsymbol{x}))^2}{\min_{\boldsymbol{x}\in\overline{D}}a_0(\boldsymbol{x})}\exp\bigg(\sum_{j=1}^\infty 3\beta_j|y_j|\bigg),\label{eq:etaval}
\end{align}
\corr{where $\beta_j=\|\psi_j\|_{L^\infty}$ is defined as in Section~\ref{sec:notations}.} Then we have for all multi-indices $\bsm$ with $|\bsm|_{\infty}:=\sup_{j\geq 1}m_j\leq 1$ that
\begin{align*}
&\bigg\|\frac{\partial_\bsy^{\boldsymbol m}a(\cdot,\bsy)}{a(\cdot,\bsy)}\bigg\|_{L^\infty(D)}\leq \boldsymbol\beta^{\boldsymbol m}\quad\text{for all}~\bsy\in U_{\boldsymbol\beta},\\
&\bigg|\frac{\partial_{\boldsymbol y}^{\boldsymbol m}\eta(\boldsymbol y)}{\eta(\boldsymbol{y})}\bigg|\leq 3^{|\boldsymbol m|}\boldsymbol\beta^{\boldsymbol m}\quad\text{for a.e.}~\bsy\in U_{\boldsymbol\beta}
\end{align*}
(note that the second inequality holds almost everywhere since the derivatives of~\eqref{eq:etaval} are discontinuous on a set of measure zero)
and the recurrence relation~\eqref{EQ:dg_induc_step} takes the form
\begin{equation}\label{eq:recurrencerel}
 \| \partial_{\bsy}^{\bsnu} u_h(\cdot, \vec y) \|_{V_h} \leq C_\textup{DG} \sum_{\mathbf 0 \neq \vecm \le \bsnu} \begin{pmatrix} \bsnu \\ \vecm \end{pmatrix}(1+3^{|\bsm|})\boldsymbol\beta^{\bsm} \| \partial_{\bsy}^{\bsnu-\bsm} u_h(\cdot, \vec y) \|_{V_h},
\end{equation}
where we used Theorem \ref{TH:stability_derivative} with $C_\textup{DG} > 0$ which is independent of $\bsy$. In this case, the assumptions of Lemma~\ref{LEM:dg_stability} are satisfied since we use NIPG while the assumptions of Lemma~\ref{LEM:norm_equiv} are satisfied by construction, with $\tilde\tau=1$. In the ensuing analysis, we only consider the first order parametric regularity for simplicity.

\begin{lemma} Let assumptions {\rm (L1)}--{\rm (L5)} hold. It holds for all $\bsnu\in\mathscr F$ with $0\neq |\bsnu|_\infty\leq 1$ and a.e.~$\bsy\in U_{\boldsymbol \beta}$ that
\begin{align}
\|\partial^\bsnu_\bsy u_h(\cdot,\bsy)\|_{V_h}\leq C_{\rm DG}^{|\bsnu|}4^{|\bsnu|}\Lambda_{|\bsnu|}\boldsymbol\beta^{\bsnu}\|u_h(\cdot,\bsy)\|_{V_h},\label{eq:lognormalinductive}
\end{align}
where the sequence of \emph{ordered Bell numbers} $(\Lambda_k)_{k=0}^\infty$ is defined by the recursion
$$
\Lambda_0:=1\quad\text{and}\quad \Lambda_k:=\sum_{\ell=1}^k \binom{k}{\ell}\Lambda_{k-\ell},\quad k\geq 1.
$$
\end{lemma}
\begin{proof} We prove the claim by induction with respect to the modulus of $\bsnu$. The base step is resolved by observing that, for any $\bsnu=\boldsymbol e_k$, it follows from~\eqref{eq:recurrencerel} that
\begin{align*}
\|\partial^{\boldsymbol e_k}_\bsy u_h(\cdot,\bsy)\|_{V_h}&\leq C_{\rm DG}(1+3)\boldsymbol\beta^{\boldsymbol e_k}\|u_h(\cdot,\bsy)\|_{V_h}\\
&= C_{\rm DG}4^{|\boldsymbol e_k|}\Lambda_{|\boldsymbol e_k|}\corr{\boldsymbol\beta^{\boldsymbol e_k}}\|u_h(\cdot,\bsy)\|_{V_h},
\end{align*}
as desired.

Next, let $\bsnu\in\mathscr F$ be such that $0\neq |\bsnu|_\infty\leq 1$ and suppose that the claim holds for all multi-indices with order $< |\bsnu|$. Then~\eqref{eq:recurrencerel} implies that
\begin{align*}
&\|\partial^{\bsnu}_\bsy u_h(\cdot,\bsy)\|_{V_h}\\&\leq C_{\rm DG}\boldsymbol\beta^{\bsnu}\|u_h(\cdot,\bsy)\|_{V_h}\sum_{\mathbf 0\neq\bsm\leq \bsnu}\binom{\bsnu}{\bsm}(1+3^{|\bsm|}) C_{\rm DG}^{|\bsnu-\bsm|}4^{|\bsnu-\bsm|}\Lambda_{|\bsnu-\bsm|}\\
&\leq C_{\rm DG}^{|\bsnu|}4^{|\bsnu|}\boldsymbol\beta^{\bsnu}\|u_h(\cdot,\bsy)\|_{V_h}\sum_{\mathbf 0\neq \bsm\leq \bsnu}\binom{\bsnu}{\bsm}\Lambda_{|\bsnu-\bsm|},
\end{align*}
where we used $C_{\rm DG}\geq 1$ and the inequality $(1+3^{\ell})4^{n-\ell}\leq 4^n$ for all $1\leq \ell\leq n$. The claim follows by simplifying the remaining sum using the generalized Vandermonde identity:
\begin{align*}
\sum_{\mathbf 0\neq \bsm\leq \bsnu}\binom{\bsnu}{\bsm}\Lambda_{|\bsnu-\bsm|}&=\sum_{\ell=1}^{|\bsnu|}\Lambda_{|\bsnu|-\ell}\sum_{|\boldsymbol m|=\ell}\binom{\bsnu}{\bsm}=\sum_{\ell=1}^{|\bsnu|}\Lambda_{|\bsnu|-\ell}\binom{|\bsnu|}{\ell}=\Lambda_{|\bsnu|}.
\end{align*}
This concludes the proof.
\end{proof}

\begin{corollary} Let assumptions {\rm (L1)}--{\rm (L5)} hold. It holds for all $\bsnu\in\mathscr F$ with $0\neq |\bsnu|_\infty\leq 1$ and a.e.~$\bsy\in U_{\boldsymbol \beta}$ that
\begin{equation*}
\|\partial_{\bsy}^{\bsnu}u_h(\cdot,\bsy)\|_{V_h}\leq C_\textup{DG}^{|\bsnu|}4^{|\bsnu|}\frac{|\bsnu|!}{(\log 2)^{|\bsnu|}}\boldsymbol\beta^{\bsnu}\frac{\sigma}{\alpha a_{\min}(\bsy)^{1/2}}\|f\|_{L^2}.
\end{equation*}
\end{corollary}
\begin{proof} The claim follows by plugging the {\em a priori} bound~\eqref{EQ:dg_stability} into~\eqref{eq:lognormalinductive} and using the well-known estimate~\cite[Lem.~A.3]{bell}:
 \begin{equation*}
  \Lambda_{|\bsnu|}\leq \frac{|\bsnu|!}{(\log 2)^{|\bsnu|}},
 \end{equation*}
 which yields the assertion.
\end{proof}

\subsection{Derivation of QMC error in the lognormal setting}\label{sec:qmcweight}
In analogy to Section~\ref{sec:unifqmc}, the dimensionally-truncated solution is given by
$$
u_{s,h}(\cdot,\bsy):=u_h(\cdot,(y_1,\ldots,y_s,0,0,\ldots)),\quad \bsy\in \mathbb R^s.
$$
\begin{theorem}\label{thm:qmcweight2}
Let assumptions {\rm (L1)}--{\rm (L5)} hold and let $\eta(\bsy)$ be chosen as in~\eqref{eq:etaval}. There exists a generating vector constructed by the CBC algorithm such that
\begin{align*}
&\sqrt{\mathbb E_{\boldsymbol\Delta}\bigg\|\int_{\mathbb R^s}u_{s,h}(\cdot,\bsy)\prod_{j=1}^s\corr{\varphi}(y_j)\,{\rm d}\bsy-Q_{\rm ran}^{\varphi}(u_{s,h})\bigg\|_{L^2(D)}^2}\\
&=\begin{cases}
\mathcal O(n^{-1/p+1/2})&\text{if}~p\in (2/3,1),\\
\mathcal O(n^{-1+\varepsilon})~\text{for arbitrary}~\varepsilon\in(0,1/2)&\text{if}~p\in (0,2/3],
\end{cases}
\end{align*}
where the implied coefficient is independent of $s$ in both cases, when the weights $(\gamma_{\setu})_{\setu\subseteq\{1:s\}}$ in Lemma~\ref{lemma:affineqmc} are chosen to be
$$
\gamma_{\setu}=\bigg(|\setu|!\prod_{j\in\setu}\frac{\beta_j}{2(\log 2)\exp(\beta_j^2/2)\Phi(\beta_j)\sqrt{(\alpha_j-\beta_j)\corr{\varrho}_j(\lambda)}}\bigg)^{2/(1+\lambda)},
$$
where
$$
\alpha_j=\frac12\bigg(\beta_j+\sqrt{\beta_j^2+1-\frac{1}{2\lambda}}\bigg)
$$
and
$$
\lambda=\begin{cases}
\frac{p}{2-p}&\text{if}~p\in (2/3,1),\\
\frac{1}{2-2\varepsilon}~\text{for arbitrary}~\varepsilon\in(0,1/2)&\text{if}~p\in (0,2/3].
\end{cases}%
$$
\end{theorem}
\begin{proof} We note  that
$$
\frac{1}{a_{\min}(\bsy)^{1/2}}\leq \frac{1}{\min_{\bsx\in \overline{D}}a_0(\bsx)^{1/2}}\prod_{j\geq 1}\exp(\tfrac12 \beta_j|y_j|)\quad\text{for all}~\bsy\in U_{\boldsymbol\beta},
$$
which yields the upper bound
$$
\bigg\|\frac{\partial^{|\setu|}}{\partial \bsy_{\setu}}u_{s,h}(\cdot,\bsy)\bigg\|_{V_h}^2\lesssim {\widehat C_{\rm DG}}^{2|\setu|}\frac{(|\setu|!)^2}{(\log 2)^{2|\setu|}}\bigg(\prod_{j\in\setu}\beta_j^2\bigg)\prod_{j\geq 1}\exp(\beta_j|y_j|),
$$
where $\widehat C_{\rm DG}:=4C_{\rm DG}$.

Since our definition of $\eta(\bsy)$ ensures that $\eta(\bsy)\geq a_{\min}(\bsy)$, we obtain
\begin{align*}
&\int_D \|u_{s,h} (\bsx,\cdot)\|_{s,\boldsymbol\gamma}^2\,{\rm d}\bsx \\
&=\sum_{\setu\subseteq\{1:s\}}\frac{1}{\gamma_{\setu}}\int_{\mathbb R^{|\setu|}}\int_D \bigg(\int_{\mathbb R^{s-|\setu|}}\frac{\partial^{|\setu|}}{\partial \bsy_{\setu}}u(\bsx,\bsy)\prod_{j\not\in\setu}\corr{\varphi}(y_j)\,{\rm d}\bsy_{-\setu}\bigg)^2\prod_{j\in\setu}\varpi_j^2(y_j)\,{\rm d}\bsy_{\setu}\\
&\leq \sum_{\setu\subseteq\{1:s\}}\frac{1}{\gamma_{\setu}}\int_{\mathbb R^{s}}\bigg\|\frac{\partial^{|\setu|}}{\partial \bsy_{\setu}}u(\cdot,\bsy)\bigg\|_{L^2(D)}^2\prod_{j\not\in\setu}\corr{\varphi}(y_j)\prod_{j\in\setu}\varpi_j^2(y_j)\,{\rm d}\bsy\\
&\leq \sum_{\setu\subseteq\{1:s\}}\frac{1}{\gamma_{\setu}}\int_{\mathbb R^{s}}{C_{\rm Poin}^{V_h}(\bsy)}^2\bigg\|\frac{\partial^{|\setu|}}{\partial \bsy_{\setu}}u(\cdot,\bsy)\bigg\|_{V_h}^2\prod_{j\not\in\setu}\corr{\varphi}(y_j)\prod_{j\in\setu}\varpi_j^2(y_j)\,{\rm d}\bsy\\
&\lesssim \sum_{\setu\subseteq\{1:s\}}\frac{{\widehat C_{\rm DG}}^{2|\setu|}}{\gamma_{\setu}}(|\setu|!)^2\bigg(\prod_{j\in \setu}\frac{\beta_j}{\log 2}\bigg)^2\\
&\qquad\times \int_{\mathbb R^{s}}\prod_{j=1}^s\exp\big(2\beta_j|y_j|)\prod_{j\not\in\setu}\corr{\varphi}(y_j)\prod_{j\in\setu}\varpi_j^2(y_j)\,{\rm d}\bsy\corr{,}
\end{align*}
\corr{where we used $C_{\rm Poin}^{V_h}(\bsy)\leq \frac{\sigma}{\sqrt{a_{\min}(\bsy)}}$.} Thus
 \begin{align*}
 &\int_D \|u_{s,h}(\bsx,\cdot)\|_{s,\boldsymbol\gamma}^2\,{\rm d}\bsx\\
&\leq \sum_{\setu\subseteq\{1:s\}}\frac{{\widehat C_{\rm DG}}^{2|\setu|}}{\gamma_{\setu}}(|\setu|!)^2\bigg(\prod_{j\in \setu}\frac{\beta_j}{\log 2}\bigg)^2\\
&\qquad\times \bigg(\prod_{j\not\in\setu}\underset{\leq 2\exp(2\beta_j^2)\Phi(2\beta_j)}{\underbrace{\int_{\mathbb R}\exp(2 \beta_j|y_j|)\corr{\varphi}(y_j)\,{\rm d}y_j}}\bigg)\bigg(\prod_{j\in\setu}\int_{\mathbb R}\exp(2\beta_j|y_j|)\varpi_j^2(y_j)\,{\rm d}y_j\bigg)\\
&\leq \sum_{\setu\subseteq\{1:s\}}\frac{{\widehat C_{\rm DG}}^{2|\setu|}}{\gamma_{\setu}}(|\setu|!)^2\bigg(\prod_{j=1}^s2\exp(2\beta_j^2)\Phi(2\beta_j)\bigg)\\
&\qquad\times\bigg(\prod_{j\in\setu}\frac{\beta_j^2}{2(\log2)^2\exp(2\beta_j^2)\Phi(2\beta_j)}\int_{\mathbb R}\exp(2\beta_j|y_j|)\varpi_j^2(y_j)\,{\rm d}y_j\bigg).
\end{align*}
Recalling that $\varpi_j(y)=\exp(\corr{-}\alpha_j|y|)$, the remaining integral is bounded as long as we choose
\begin{align}
\alpha_j>\beta_j.\label{eq:alphacond}
\end{align}

In fact, if~\eqref{eq:alphacond} holds, then
$$
\int_{\mathbb R}\exp(2\beta_j|y_j|)\varpi_j^2(y_j)\,{\rm d}y_j=\frac{1}{\alpha_j-\beta_j}
$$and we obtain
 \begin{align*}
 \int_D \|u_{s,h}(\bsx,\cdot)\|_{s,\boldsymbol\gamma}^2\,{\rm d}\bsx &\leq \sum_{\setu\subseteq\{1:s\}}\frac{{\widehat C_{\rm DG}}^{2|\setu|}}{\gamma_{\setu}}(|\setu|!)^2\bigg(\prod_{j=1}^s2\exp(2\beta_j^2)\Phi(2\beta_j)\bigg)\\
&\qquad\times\bigg(\prod_{j\in\setu}\frac{\beta_j^2}{2(\log2)^2\exp(2\beta_j^2)\Phi(2\beta_j)(\alpha_j-\beta_j)}\bigg).
 \end{align*}
The remainder of the argument is completely analogous to the derivation presented for the continuous setting in~\cite{log}: \corr{the weights $\gamma_{\setu}$ enter the expression for the upper bound in the same manner as in the continuous setting, leading us to conclude that} the error criterion is minimized by setting
$$
\alpha_j=\frac12\bigg(\beta_j+\sqrt{\beta_j^2+1-\frac{1}{2\lambda}}\bigg)
$$
and choosing the weights
$$
\gamma_{\setu}=\bigg(|\setu|!\prod_{j\in\setu}\frac{\beta_j}{2(\log 2)\exp(\beta_j^2/2)\Phi(\beta_j)\sqrt{(\alpha_j-\beta_j)\corr{\varrho}_j(\lambda)}}\bigg)^{2/(1+\lambda)},
$$
with $\lambda$ chosen as in~\eqref{eq:lambdaaffine} and $p$ as in (L3).
\end{proof}

\corr{{\em Remark.} Let $\mathcal G\!:V_h\to \mathbb R$ be a bounded linear functional. Similarly to the affine and uniform setting, it can be shown that there exists a generating vector constructed by the CBC algorithm using the weights specified in Theorem~\ref{thm:qmcweight2} such that
$$
\sqrt{\mathbb E_{\boldsymbol\Delta}\bigg|\int_{\mathbb R^s}\mathcal G(u_{s,h}(\cdot,\bsy))\prod_{j=1}^s\corr{\varphi}(y_j)\,{\rm d}\bsy-Q_{\rm ran}(\mathcal G(u_{s,h}))\bigg|^2}=\mathcal O(n^{\max\{-1/p+1/2,-1+\varepsilon\}}),
$$
where the implied coefficient is independent of the dimension $s$ with arbitrary $\varepsilon\in(0,1/2)$.
}

\section{Numerical results}\label{SEC:numerics}
%
We consider \eqref{eq:uncertainpde} with $f(\vec x) = x_1$ in $D = (0,1)^2$ and investigate the errors in the means of the numerical approximations to the unknown $u$. For the affine case, we set $U = [-\tfrac{1}{2},\frac{1}{2}]^{\mathbb N}$ and truncate the series expansion for the input random coefficient into $s=100$ terms, i.e.,
\begin{equation*}
 a_\textup{affine}(\vec x,\vec y) = 5 + \sum_{j=1}^{100} \frac{y_j}{(k_j^2 + \ell_j^2)^{1.3}} \sin(k_j \pi x_1) \sin(\ell_j \pi x_2),
\end{equation*}
where $(k_j, \ell_j)_{j \ge 1}$ is an ordering of elements of $\mathbb Z_+ \times \mathbb Z_+$ such that the sequence $(\|\psi_j\|_{L^\infty(D)})_{j \ge 1}$ is not increasing. In the lognormal case, we define $U = \mathbb R^{\mathbb N}$ and consider the dimensionally-truncated coefficient with $s=100$ terms
\begin{equation*}
 a_\textup{lognormal}(\vec x, \vec y) = \exp\left( \sum_{j=1}^{100} \frac{y_j}{(k_j^2 + \ell_j^2)^{1.3}} \sin(k_j \pi x_1) \sin(\ell_j \pi x_2) \right)
\end{equation*}
with an analogously defined sequence $(k_j, \ell_j)_{j \ge 1}$. In this case, we have that $\|\psi_j\|_{L^\infty(D)}\sim j^{-1.3}$ and the expected convergence independently of the dimension is $\mathcal O(n^{-0.8+\varepsilon})$, $\varepsilon>0$.

We use in all experiments an off-the-shelf generating vector~\cite[lattice-39101-1024-1048576.3600]{lattice} using a total amount of 16384, 32768, 65536, 131072, 262144, and 524288 cubature points with $R=16$ random shifts. \corr{Although this generating vector has not been obtained using the CBC algorithm with the weights derived in Theorems~\ref{thm:qmcweight} and~\ref{thm:qmcweight2}, we found this off-the-shelf generating vector to perform well yielding the optimal rate of convergence. Thus, this confirms our analytical findings without excluding that there might be even better lattice rules.} In practice, the off-the-shelf lattice rule has comparable performance to tailored lattice rules. The finite element discretization uses discontinuous Galerkin methods (the implementation of which is based on the `Finite Element Simulation Toolbox for UNstructured Grids: FESTUNG' \cite{ReuterHRFAK20}) on a grid of mesh size $\tfrac{1}{16}$. Additionally, the DG methods are compared to a conforming finite element implementation on $30 \times 30$ elements.

Thus, the spatial grid of the discontinuous Galerkin method is significantly coarser than the conforming finite element grid. We do so to underline that it is generally considered `unfair' to compare discontinuous Galerkin and conforming finite elements on the same grid, since DG usually has many more degrees of freedom. There are several ways to compensate this imbalance (count degrees of freedom, number non-zero entries in the stiffness matrix, \dots). However, we will see that the accuracy of the DG and conforming element methods is of minor importance in the numerical experiments. Thus, we skip a detailed discussion about a fair comparison.

\begin{figure} \centering
\begin{tikzpicture}
  \begin{axis}[
   xmode   = log,
   ymode   = log,
   ylabel  = R.M.S. error,
   xlabel  = Total number of cubature points,
   height  = 0.45\linewidth,
   width   = 0.45\linewidth,
   domain  = 16*1000:16*33000,
   samples = 5
   ]
   \addplot[red, only marks, mark=otimes] table {matlab_output/affine_symparam-1_eta100.txt};
   \addplot[red, mark=none] {0.000252565907890 * x^(-1.099034898308349)};
   \addplot[blue, only marks, mark=otimes] table {matlab_output/affine_symparam1_eta10.txt};
   \addplot[blue, mark=none] {0.000243210149904 * x^(-1.098598996123883)};
   \addplot[green, only marks, mark=otimes] table {matlab_output/affine_symparam-1_eta10.txt};
   \addplot[green, mark=none] {0.000951130781053 * x^(-0.353534266830121)};
   \addplot[black, only marks, mark=otimes] table {matlab_output/affine_fem.txt};
   \addplot[black, dashed, mark=none] {0.000251076120179 * x^(-1.098656383537929)};
  \end{axis}
 \end{tikzpicture}
 \hfill
 \begin{tikzpicture}[spy using outlines={circle, magnification=5, size=0.3cm, connect spies}]
  \begin{axis}[
   xmode   = log,
   ymode   = log,
   ylabel  = R.M.S. error,
   xlabel  = Total number of cubature points,
   height  = 0.45\linewidth,
   width   = 0.45\linewidth,
   domain  = 16*1000:16*33000,
   samples = 5
   ]
   \addplot[red, only marks, mark=x] table {matlab_output/affine_symparam-1_eta100.txt};
   \addplot[red, mark=none] {0.000252565907890 * x^(-1.099034898308349)};
   \addplot[blue, only marks, mark=x] table {matlab_output/affine_symparam1_eta10.txt};
   \addplot[blue, mark=none] {0.000243210149904 * x^(-1.098598996123883)};
   \addplot[black, only marks, mark=+] table {matlab_output/affine_fem.txt};
   \addplot[black, dashed, mark=none] {0.000251076120179 * x^(-1.098656383537929)};
   \begin{scope}
    \spy[black,size=2cm] on (3.2,1.33) in node [fill=white] at (3.2,3.2);
   \end{scope}
  \end{axis}
 \end{tikzpicture}
 \caption{Root mean squared error for the affine case for first order polynomial DG methods. The NIPG method with $\eta = 10$ is depited blue, while the SIPG method with $\eta = 100$ is depicted red and the SIPG method with $\eta = 10$ is depicted green, and the conforming finite element solution is black. The green graph is omitted in the right picture.}\label{FIG:affine results}
\end{figure}
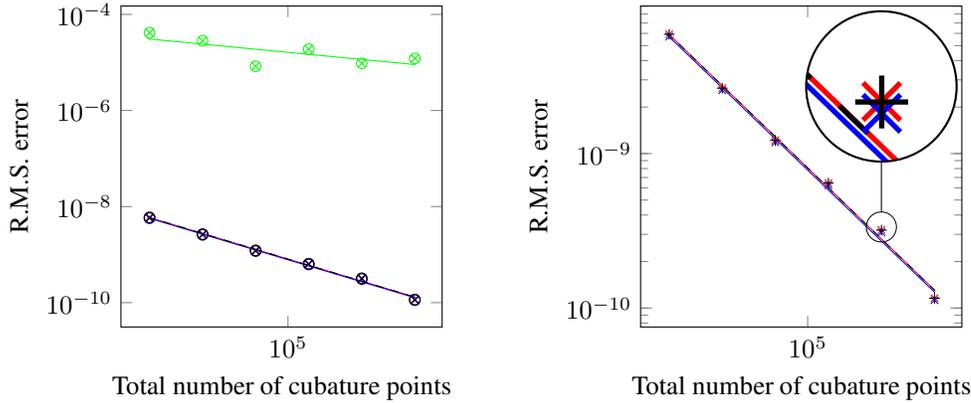
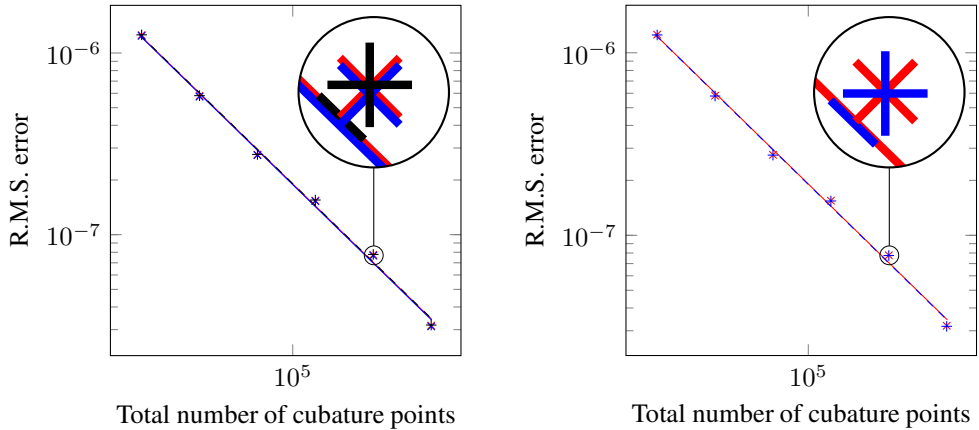
\begin{figure} \centering
 \begin{tikzpicture}[spy using outlines={circle, magnification=8, size=0.3cm, connect spies}]
  \begin{axis}[
   xmode   = log,
   ymode   = log,
   ylabel  = R.M.S. error,
   xlabel  = Total number of cubature points,
   height  = 0.48\linewidth,
   width   = 0.48\linewidth,
   domain  = 16*1000:16*33000,
   samples = 5
   ]
   \addplot[red, only marks, mark=x] table {matlab_output/lognormal_symparam-1_eta10.txt};
   \addplot[red, mark=none] {0.027367855497638 * x^(-1.031589222967223)};
   \addplot[blue, only marks, mark=x] table {matlab_output/lognormal_symparam1_eta10.txt};
   \addplot[blue, mark=none] {0.027287093535884 * x^(-1.032158554376227)};
  \addplot[black, only marks, mark=+] table {matlab_output/lognormal_fem.txt};
  \addplot[black, dashed, mark=none] {0.027201590399422 * x^(-1.030671255419436)};
  \begin{scope}
    \spy[black,size=2cm] on (3.5,1.33) in node [fill=white] at (3.5,3.5);
  \end{scope}
  \end{axis}
 \end{tikzpicture}
 \hfill
 \begin{tikzpicture}[spy using outlines={circle, magnification=8, size=0.3cm, connect spies}]
  \begin{axis}[
   xmode   = log,
   ymode   = log,
   ylabel  = R.M.S. error,
   xlabel  = Total number of cubature points,
   height  = 0.48\linewidth,
   width   = 0.48\linewidth,
   domain  = 16*1000:16*33000,
   samples = 5
   ]
  \addplot[red, only marks, mark=x] table {matlab_output/lognormal_p2_symparam-1_eta10.txt};
  \addplot[red, mark=none] {0.026802473113557 * x^(-1.029429354206398)};
  \addplot[blue, only marks, mark=+] table {matlab_output/lognormal_p2_symparam1_eta10.txt};
  \addplot[blue, dashed, mark=none] {0.026885750165130 * x^(-1.030076355729698)};
  \begin{scope}
    \spy[black,size=2cm] on (3.5,1.33) in node [fill=white] at (3.5,3.5);
  \end{scope}
  \end{axis}
 \end{tikzpicture}
 \caption{Root mean squared error for the lognormal case. The NIPG method with $\eta = 10$ is depicted in red, the SIPG method with $\eta = 10$ is depicted in blue, and the conforming finite element solution is black. The left picture illustrates DG and conforming finite elements for locally linear polynomial approximation, while the second picture shows only DG for a locally quadratic approximation.}\label{FIG:lognormal results}
\end{figure}

Figure \ref{FIG:affine results} shows the numerical results for the affine case and locally linear DG or conforming finite element approximation. We recognize that the green plot for the SIPG method with $\eta = 10$ does not show the expected convergence behavior (with respect to the number of cubature points). This nicely resembles our analysis, which highlights that SIPG needs a large enough $\eta$ to work stably, while the NIPG method is stable for all $\eta > 0$. We clearly see that SIPG (with $\eta = 100$, red), NIPG (with $\eta = 10$, blue), and conforming linear finite elements (black) all show a very similar convergence rate (between 1.098 and 1.1) in the number of cubature points indicating that all three methods work similarly well in the affine case.

In Figures~\ref{FIG:affine results} and~\ref{FIG:lognormal results} we recognize that the error plots of all stable methods are almost identical. Thus we conclude that the cubature error dominates the discretization error and hence the particular choice of the discretization appears to be almost irrelevant. 

Figure \ref{FIG:lognormal results} deals with the lognormal case. The left picture uses linear approximations of SIPG, NIPG, and the conforming finite element method, while the right picture shows the results for second order SIPG and NIPG. In the left picture, we see that, again, all three methods work fine, and similarly well. Their convergence rates are approx.\ 1.03 with only insignificant differences. Interestingly, we see similar convergence rates and absolute errors for the second order SIPG and NIPG approximations. Thus, we deduce that the influence of the accuracy of the conforming finite element or DG approximation of the PDE does not significantly influence the convergence behavior (not even the constants) of the QMC method in our cases.

\corr{Notably, the NIPG method is robust concerning $\eta$, and we do not see visual differences in the error plots if we choose $\eta(\vec y)$ as described in the remark in Section \ref{sec:stability}.}

\section{Conclusion}
We have designed QMC cubatures using DG discretizations of an elliptic PDE with a random coefficient. Our analytical and numerical findings are consistent with the results derived from the conforming approximations in the literature.

\corr{In this study, we assumed the data to allow for solutions $u \in H^2$ concerning the spatial variable. At the cost of some technical but standard extensions of the DG analysis, we can extend our results to the case that $u \in H^{3/2+\epsilon}$ in space. Moreover, using discrete gradients and lifting operators, we believe our results can be transferred to the case of $u$ having minimal spatial regularity. We dropped extensive discussions on spatial regularity, as detailed in the DG literature. For QMC methods, the advantages of DG are its more regular execution patterns and the additional local mass conservation. In that sense, it preserves more physics than continuous finite element methods. However, we are not aware of any indications that DG works under milder regularity assumptions than continuous finite elements.}

\section*{Acknowledgements}%

\corr{A.\ Rupp has been supported by the Academy of Finland's grant number 350101 \emph{Mathematical models and numerical methods for water management in soils}, grant number 354489 \emph{Uncertainty quantification for PDEs on hypergraphs}, grant number 359633 \emph{Localized orthogonal decomposition for high-order, hybrid finite elements}, Business Finland's project number 539/31/2023 \emph{3D-Cure: 3D printing for personalized medicine and customized drug delivery}, and the Finnish \emph{Flagship of advanced mathematics for sensing, imaging and modelling}, decision number 358944.}

\bibliography{references}
\bibliographystyle{siam}

\end{document}